\documentclass[a4paper,11pt]{article}

\usepackage{amssymb}
\usepackage{amsfonts,dsfont}
\usepackage{amsmath}
\usepackage{amsthm}
\usepackage{cite}
\usepackage{enumitem}

\usepackage[colorlinks=true, pdftitle={Extension problems for the logarithmic Laplacian}, 
pdfauthor={Huyuan Chen, Daniel Hauer, Tobias Weth}]{hyperref}

\newcommand{\s}{s}

\usepackage{xcolor}

\newcommand{\R}{\mathbb{R}}

\newcommand{\N}{\mathbb{N}}

\newcommand{\cuad}{{\sqcap\kern-.68em\sqcup}}

\newtheorem{definition}{Definition}[section]
\newtheorem{theorem}[definition]{Theorem}
\newtheorem{proposition}[definition]{Proposition}

\newtheorem{lemma}[definition]{Lemma}
\newtheorem{corollary}[definition]{Corollary}
\newtheorem{remark}[definition]{Remark}
\newcommand{\bremark}{\begin{remark} \em}
\newcommand{\eremark}{\end{remark} }

\newcommand{\Lra}{\Longrightarrow}

\newcommand{\loglap}{L_{\text{\tiny $\Delta \,$}}}

\renewcommand{\div}{\,{\rm div}\,}
\renewcommand{\phi}{\varphi}

\newcommand{\cA}{{\mathcal A}}
\newcommand{\cAG}{{\mathcal A \mathcal G}}

\newcommand{\cE}{{\mathcal E}}

\newcommand{\cL}{{\mathcal L}}

\newcommand{\supp}{{\rm supp}}

\newcommand{\LNdo}{\,\cL^N \!\! \otimes \! \delta_{0}}
\newcommand{\LNdoo}{\,\cL^N \!\! \otimes \! \delta_{(0,0)}}

\newcommand{\eps}{\varepsilon}

\title{Extension problems for the logarithmic Laplacian}

\begin{document}
\begin{center}
{\large\bf An extension problem for the logarithmic Laplacian} 

 \bigskip

 {\small  Huyuan Chen\footnote{chenhuyuan@yeah.net} \qquad Daniel Hauer\footnote{Daniel Hauer <daniel.hauer@sydney.edu.au}\qquad    Tobias Weth\footnote{weth@math.uni-frankurt.de}
}
\bigskip

{\small  $ ^1$Department of Mathematics, Jiangxi Normal University, Nanchang,\\
Jiangxi 330022, PR China\\[3mm]

$ ^2$The University of Sydney, School of Mathematics and Statistics, NSW 2006, Australia\\[3mm]

$ ^3$Goethe-Universit\"{a}t Frankfurt, Institut f\"{u}r Mathematik, Robert-Mayer-Str. 10,\\
D-60629 Frankfurt, Germany
 } \\[6mm]

\end{center}

\begin{abstract}
  The logarithmic Laplacian on the (whole) $N$-dimensional Euclidean
  space is defined as the first variation of the fractional
  Laplacian of order $2s$ at $s=0$ or, alternatively, as a singular Fourier integral
  operator with logarithmic symbol. While this operator has attracted fastly growing attention in recent years due to
  its relevance in the study of order-dependent problems, a characterization   
  via a local extension problem
  on the $(N+1)$-dimensional upper half-space in the spirit of the Cafferelli-Sivestre extension for the fractional Laplacian has been missing so far. 
  In this paper, we establish such a characterization. More precisely, we show
  that, up to a multiplicative constant, the logarithmic Laplacian
  coincides with the boundary-value operator associated with a weighted
  second-order operator on the upper half-space, which maps
  inhomogeneous Neumann data to a Robin boundary-value of the
  corresponding distributional solution with a singular excess term. This
  extension property of the logarithmic Laplacian leads to a new energy
  functional associated with this operator. By doubling the
  extension-variable, we show that distributional solutions of the
  extension problem are actually harmonic in the $(N+2)$-dimensional
  Euclidean space away from the boundary. As an application of these
  results, we establish a weak unique continuation principle for the
  (stationary) logarithmic Laplace equation.
\end{abstract}

\setcounter{equation}{0}
\section{Introduction}

In recent years, the interest in the modeling of nonlocal phenomena via hypersingular
integral operators and their study from a PDE perspective has increased significantly. For an overview of these developments, see
e.g. the recent books \cite{bucur-valdinoci},\cite{kuusi-palatucci} and
the references therein. The most prominent role in this context is
played by the fractional Laplacian $(-\Delta)^\s$ of order $2s$, which
for $\s\in(0,1)$ is given by
 \begin{equation}\label{fl 1}
 (-\Delta)^\s  u(x)=c_{N,\s} \lim_{\epsilon\to0^+} \int_{\R^N\setminus
   B_\epsilon(x) }\frac{u(x)-u(y)}{|x-y|^{N+2\s}}  dy  \qquad \text{for $u \in C^\infty_c(\R^N)$.}
\end{equation}
Here
$c_{N,\s}=2^{2\s}\pi^{-\frac{N}{2}}\s\frac{\Gamma(\frac{N+2\s}2)}{\Gamma(1-\s)}$
is a normalization constant which is chosen such that $(-\Delta)^s$
corresponds to the Fourier symbol $\xi \mapsto |\xi|^{2s}$, i.e. we have
\begin{equation}
 \label{eq:Fourier-delta-1} 
 \mathcal{F}((-\Delta)^\s u)(\xi) = |\xi|^{2\s}\hat u (\xi)\qquad  
\text{for $u \in C^2_c(\R^N)$ and $\xi \in \R^N$.}
\end{equation}
Using this property, it is easy to see that 
 \begin{equation} \label{eq:limit-behaviour}
\lim_{\s\to0^+}  (-\Delta)^\s  u(x) = u(x)
\quad\  {\rm and}\quad\ \quad\  \lim_{\s\to1^-}(-\Delta)^\s u(x)=-\Delta
u(x)
 \end{equation}
 for $u\in C^2_c(\R^N)$. Motivated by this observation, the first and
 third author of this paper introduced in \cite{CW18} the
 \emph{logarithmic Laplacian} $\loglap$ as a first order correction in
 the first limit in (\ref{eq:limit-behaviour}). More precisely, if
 $u\in C_c^{\alpha}(\R^N)$ for some $\alpha>0$, the function $\loglap u$
 is uniquely given, for $x \in \R^N$, by the expansion
  \begin{equation} \label{eq:limit-behaviour1} (-\Delta)^\s u(x) = u(x)
    + s \loglap u (x) + o(\s) \quad{\rm as}\ \, \s\to0^+,
 \end{equation}
 or equivalently, by
 \begin{displaymath}
\frac{d}{ds}(-\Delta)^su\Big|_{s=0}=\loglap u.
\end{displaymath}
So the logarithmic Laplacian $\loglap$ can be seen as the derivative of
operator valued map $s \mapsto (-\Delta)^s$ at $s=0$. Moreover, as noted in
\cite[Theorem 1.1]{CW18}, \eqref{eq:Fourier-delta-1} leads to
\begin{displaymath}
  \mathcal{F}(\loglap u)(\xi) = 2\ln|\xi|\hat u (\xi)\qquad 
\qquad  \text{for $\xi \in \R^N$ if $u \in C^\alpha_c(\R^N)$ for some $\alpha>0$.}
\end{displaymath}
Hence $\loglap u$ is a (weakly) singular Fourier integral operator
associated with the symbol $2\ln|\cdot|$. In addition, it is shown in
\cite[Theorem 1.1]{CW18} that $\loglap u$ admits the integral
representation
\begin{equation}\label{eq:log-laplace}
\loglap u(x)=c_N P.V.\int_{B_1(0)}\frac{u(x)-u(x+y)}{|y|^{N}}\
dy - c_N\int_{\R^N\setminus B_1(0)}\frac{u(x+y)}{|y|^{N}}\ dy+\rho_N\,u(x),  
\end{equation}
where 
\begin{equation}
  \label{eq:def-c-n-rho-n}
c_N=\frac{\Gamma(\frac{N}{2})}{\pi^{N/2}}= \frac{2}{\omega_{_N}}, 
\qquad\qquad \rho_N :=2\ln(2)+\psi(\frac{N}{2}) - \gamma,   
\end{equation}  
$\gamma=-\Gamma'(1)$ is the Euler Mascheroni constant, and
$\psi=\Gamma'/\Gamma$ is the Digamma function.  Here and in the
following, $\omega_{_N}:=|\mathbb{S}^{N-1}|$ denotes the
$(N-1)$-dimensional volume of the unit sphere in $\R^N$. In fact, for
the value of $\loglap u(x)$ to be well defined
by~\eqref{eq:log-laplace}, it suffices to assume that $\phi$ is Dini
continuous at $x\in \R^{N}$ and belongs to the space
\begin{equation}
\label{eq:def-L-1-0-space}  
  L^1_0(\R^N):= \Big\{ u \in L^1_{loc}(\R^N)\::\: 
  \int_{\R^N}(1+|x|)^{-N}|u(x)|\,dx < \infty\Big\},
\end{equation}
see \cite[Proposition 1.3]{CW18}. 

The logarithmic Laplacian $\loglap$ is a very useful operator to study
the asymptotics of the Dirichlet (eigenvalue) problem for the fractional
Laplacian $(-\Delta)^s$ for $s \to 0^+$ on open bounded Lipschitz sets
$\Omega \subset \R^N$.  In the case of Dirichlet eigenvalues and
eigenfunctions of $(-\Delta)^s$, these asymptotics have been studied in
detail in \cite{CW18} and \cite{FJW}. It has been shown in \cite{CW18}
that the first Dirichlet eigenvalue $\lambda_1^s$ of $(-\Delta)^s$ on
$\Omega$ satisfies the expansion
$\lambda_1^s = 1+ s \lambda_{1}^{\loglap}+o(s)$ as $s \to 0$, where
$\lambda_{1}^{\loglap}$ denotes the first Dirichlet eigenvalue of
$\loglap$. Moreover, the unique $L^2$-normalized positive eigenfunction
corresponding to $\lambda_1^s$ converges, as $s \to 0^+$, in
$L^2(\Omega)$ to the unique $L^2$-normalized positive Dirichlet
eigenfunction of $\loglap$. These convergence results have been refined
in \cite{FJW} and extended to higher eigenvalues and
eigenfunctions. Moreover, eigenvalue estimates and Weyl type asymptotics
for the logarithmic Laplacian have been derived in
\cite{CV,LW}.

More recently, the small order limit $s \to 0^+$ and its relationship to
$\loglap$ has also been studied in the framework of $s$-dependent {\em
  nonlinear} Dirichlet problems, see \cite{angeles-saldana,hs.saldana}.
The asymptotic study of fractional problems in this small order limit is
motivated, in particular, by order-dependent optimization problems with
small optimal order, see e.g. \cite{antil.bartels} and
\cite{sprekels.valdinoci,pellacci.verzini} for applications in image
processing and population dynamics.

It worth mentioning that the operator $\loglap$ is not only useful for
small order problems. In fact, it is shown in \cite{jarohs-saldana-weth}
that the operator $\loglap$ allows to characterize the $s$-dependence of
solution to fractional Poisson problems for the full range of exponents
$s \in (0,1)$. The logarithmic Laplacian also appears in the geometric
context of the $0$-fractional perimeter, see \cite{DNP}.  Finally, we
mention that a higher order operator expansion in the spirit of
(\ref{eq:limit-behaviour1}) is derived and investigated in \cite{C}.

Due to its relevance, it is natural to ask whether the logarithmic
Laplacian admits a representation by a local extension problem. This has
been an open question for the last years, and the purpose of the present
paper is to give an affirmative answer to this question.  The interest
in local extension problems is very natural as they pave the way for the
application of classical PDE techniques to associated nonlocal
problems. In this context, the starting point of our paper is the
so-called \emph{Caffarelli-Silvestre extension property} of the
fractional Laplacian $(-\Delta)^s$ given by the characterization
\begin{displaymath}
\label{c-s-extension-d-s-constant}  
  (-\Delta)^s=d_s\,\Lambda_s \qquad \qquad \text{with}\quad   
  d_s = 2^{2s-1} \frac{\Gamma(s)}{\Gamma(1-s)},
\end{displaymath}
 where $\Lambda_s$ denotes the \emph{Dirichlet-to-Neumann operator} 
associated with the weighted second-order differential operator
\begin{displaymath}
  \mathcal{A}_s=-\div (t^{1-2s} \nabla \cdot)
\end{displaymath}
on the half space $\R^{N+1}_{+}:=\R^N\times (0,+\infty)$. More
precisely, Caffarelli and Silvestre
\cite{Caffarelli-Silvestre} showed  that for every $u\in
C_c^{\infty}(\R^N)$, there is an \emph{$s$-harmonic extension} 
$w_s:\R^{N+1}_+ \to \R$ of $u(x)=w_s(x,0)$, $x\in \R^N$; that is, $w_s$ is a
distributional solution of the \emph{Dirichlet problem}
 \begin{equation}\label{Ext-s}
\begin{cases}
\begin{aligned}
 -\div (t^{1-2s} \nabla w_s) =  0 \quad\ &\text{in \ \  $\R^{N+1}_+$,}\\[2mm]
 w_s = u \quad\ &\text{on \ \ $\R^N= \partial \R^{N+1}_+$},  
\end{aligned}
\end{cases}
\end{equation}
where $s\in(0,1)$, and for the \emph{co-normal derivative}
\begin{displaymath}
  \Lambda_su:=-\lim_{t \to 0^+} t^{1-2s}\partial_t w_s
\end{displaymath}
one has that
\begin{equation}
  \label{eq:caffarelli-silvestre-boundary-representiation}
  (-\Delta)^s u = - d_s \lim_{t \to 0^+} t^{1-2s}\partial_t w_s
\end{equation}
with some constant $d_s>0$. The value \eqref{c-s-extension-d-s-constant}
of the constant $d_s$ was then given by Cabr\'e and Sire in
\cite{CS}). This property of $(-\Delta)^s$ provides a link between
nonlocal problems involving $(-\Delta)^su$ with a local one by extending
$u$ on $\R^N$ $s$-harmonically on $\R^{N+1}_{+}$. In particular, it
provides the PDE point of view that a symmetric $2s$-stable L\'evy
process (generated by the fractional Laplacian $(-\Delta)^s$, see
\cite{appelbaum}) coincides with the trace process of a degenerate
diffusion process in $\R^{N+1}_+$ (see~\cite{MalOst}, and for further
details to the literature see~\cite[Section~2]{HauerLee}).\medskip

As mentioned before, the purpose of the present paper is to derive a
representation of the operator logarithmic Laplacian $\loglap$ via a
local extension problem.  In what follows, we shall see that $\loglap$
admits such a representation, but the corresponding extension problem
and boundary operator are very different from the Caffarelli-Silvestre
extension problem for the fractional Laplacian. To state our main
results, we need the following definition.

 \begin{definition}
\label{algebraic-growth}   
For $d \in \N$ and a subset $\Omega \subset \R^d$, we let $\cAG(\Omega)$
denote the space of functions $w \in L^1_{loc}(\Omega)$ having at most
algebraic growth in the integral sense, i.e., the space of functions
$w \in L^1_{loc}(\Omega)$ with the property that, for some constants
$C,\sigma>0$, we have
\begin{displaymath}
 \|w\|_{L^1(B_1(x)\cap \Omega)} \le C (1+|x|)^\sigma  \qquad \text{for all $x \in \Omega$.}
\end{displaymath}
\end{definition}

Further, we recall that, by definition, a function $u: \R^N \to \R$
is Dini continuous at a point $x \in \R^N$ if
$$
\int_0^1 \frac{\omega_{u,x}(r)}{r}\,dr < \infty\qquad \text{with}\qquad  
\omega_{u,x}(r)= \sup_{y \in B_r(x)} |u(y)-u(x)|.
$$
Moreover, we call $u$ {\em uniformly Dini continuous on a subset $K \subset \R^N$} if 
\begin{equation*}
\int_0^1 \frac{\omega_{u,\text{\tiny $K$}}(r)}{r}\,dr < \infty \qquad 
\text{for $\omega_{u,\text{\tiny $K$}}(r):=
\sup \limits_{x \in K}\omega_{u,x}(r)$.}
\end{equation*}

Now, we can state the first variant of the extension property associated
with the logarithmic Laplacian $\loglap$.

 \begin{theorem}
   \label{thm:main1}
   For every $u \in L^1_0(\R^N)$, there is a unique solution
   \begin{displaymath}
     w_u \in \cAG(\R^{N+1}_+) \cap
     C^\infty(\R^{N+1}_+)
   \end{displaymath}
   of the problem
  \begin{equation}
    	\label{eq:Poisson-problem-thm:main1}
		-\div(t\,\nabla
                w_u) = 0  \quad \text{in $\R^{N+1}_+$}
    \end{equation}    
   satisfying the asymptotic boundary conditions 
   \begin{equation}
     \label{eq:thm:main1-asymptotic-1}
     -\lim_{t \to 0^+} t \partial_t w_u(\cdot,t) = u \qquad \text{in $L^1_{loc}(\R^N)$}
   \end{equation}
   and 
    \begin{displaymath}
          \label{eq:thm:main1-asymptotic-2}
        \lim_{t\to + \infty}w_u(x,t)=0 \qquad\text{for every $x \in \R^N$.}
      \end{displaymath}
Moreover, 
    the following statements hold:
    \begin{enumerate}[label=(\arabic*.)]
    \item[(i)] \label{thm:main1-claim1} The function $w_u$ can be represented via the Poisson formula
      \begin{equation}
\label{eq:poisson-formula-wu}
        w_u(x,t)=\frac{c_N}{2}\int_{\R^N}\frac{u(\tilde x)}{(|x-\tilde x|^2+t^2)^{N/2}}\,d\tilde x \qquad \text{for every $(x,t)\in \R^{N+1}_+$;}
      \end{equation}
    \item[(ii)] 
      \begin{equation}
        \label{u-limit-L-1-loc}
        \lim_{t\to 0^{+}}\frac{w_u(\cdot,t)}{\ln
          t}=- u \qquad \text{in $L^1_{loc}(\R^N)$;}
      \end{equation}
    \item [(iii)] 
      \begin{equation}
        \label{eq:characterization-loglap}
        \loglap u = 2(\ln 2-\gamma)u - 2\lim_{t \to 0^+}\Bigl(w_u(\cdot,t) + u \log
        t \Bigr) \quad \text{in the distributional sense in $\R^N$;}
      \end{equation}
    \item[(iv)] If $u$ is Dini continuous at a point $x \in \R^N$, then
      (\ref{eq:characterization-loglap}) holds in pointwise sense, i.e.
      we have
      \begin{equation}
        \label{eq:characterization-loglap-pointwise}
              \loglap u(x) = 2\,(\ln 2 - \gamma)\,u(x)
                             -  2\lim_{t \to 0^+}\Bigl(w_u(x,t) + u(x) \ln t\Bigr) 
      \end{equation}
      as a pointwise limit in $\R$.
    \end{enumerate}
\end{theorem}

The following comments to Theorem~\ref{thm:main1} are worth stating.

\begin{remark}
\label{remark-after-first-main-theorem}  
(i) Due to the characterization given by
\eqref{eq:characterization-loglap} 
(respectively, \ref{eq:characterization-loglap-pointwise}), the logarithmic Laplacian
$\loglap$ can, after subtracting a multiple of the identity, be
characterized as \emph{Neumann-to-Dirichlet operator} associated with
the weighted second-order differential operator
\begin{displaymath}
  \mathcal{A}_0=-\div (t \nabla \cdot)
\end{displaymath}
on the half space $\R^{N+1}_{+}$, and with the singular excess term
$u(x) \ln t $ as $t\to 0^+$. 

Alternatively, since for every $u\in L^1_0(\R^N)$, the extension $w_u$
attains $u$ as limit of the co-normal derivative
\eqref{u-limit-L-1-loc}, the characterization given by
\eqref{eq:characterization-loglap} (respectively,
\ref{eq:characterization-loglap-pointwise}) says that the logarithmic
Laplacian $\loglap$ coincides up to a constant multiple with the
\emph{Neumann-to-Robin operator} associated with $\mathcal{A}_0$ on
$\R^{N+1}_{+}$ and with the singular excess term $u(x) \ln t $ as
$t\to 0^+$.

(ii) If $u \in L^1_0(\R^N)$ and $w_u$ is defined by
(\ref{eq:poisson-formula-wu}), it is tempting to guess that the
convergence (\ref{eq:thm:main1-asymptotic-1}) and its counterpart
$\frac{w_u(\cdot,t)}{\ln t} \to -u$ as $t \to 0^+$ hold also in the
weighted $L^1$-space $L^1_0(\R^N)$. However, in general the functions
$w_u(\cdot,t)$ do not need to belong to the space $L^1_0(\R^N)$ in this
case, see Remark~\ref{counterexample} for a counterexample.

(iii) By definition, the distributional limit in
(\ref{eq:characterization-loglap}) means that
\begin{equation*}
  \int_{\R^N}u \loglap \phi dx =  2(\ln 2-\gamma) \int_{\R^N}u \phi\,dx - 
  2\lim_{t \to 0^+}  \int_{\R^N}\Bigl(w_u(x,t) + u \ln t\Bigr)\phi (x)
  \,dx\quad 
\:\text{for all $\phi \in C^\infty_{c}(\R^N)$.} 
\end{equation*}
In fact, we shall show that this property already holds if
$\phi \in C_c^D(\R^N)$, where $C_c^D(\R^N)$ denotes the space of
uniformly Dini continous functions on $\R^N$ with compact support. A
direct consequence of this property is the following alternative
representation of the energy associated with $\loglap$, which has been introduced in \cite{CW18} in the form 
$$
\phi  \mapsto \cE_L(\phi,\phi)=\frac{c_N}{2} \int_{|x-\tilde x|<1} \frac{(\phi(x)-\phi(\tilde x))^2}{ |x-\tilde x|^N} dx d\tilde x  -\frac{c_N}{2} \int_{|x-\tilde x|\geq 1} \frac{ \phi(x)\phi(\tilde x) }{ |x-\tilde x|^N}  \,dxd\tilde x+\frac{\rho_N}{2} \int_{\R^N}   \phi(x)^2  \,dx.
$$
\end{remark}

\begin{corollary}
  \label{sec:energy-of-loglap}
  For every $\phi \in C^D_c(\R^N)$, one has that
  \begin{displaymath}
  \cE_{L}(\phi,\phi) = \int_{\R^N}\phi \loglap \phi dx =   2(\ln 2-\gamma)\|\phi\|_{L^2(\R^N)}^2 -
  2\lim_{t \to 0^+} 
 \int_{\R^N}\Bigl(\phi(x) w_\phi(x,t) + \phi^2(x) \ln t \Bigr) \,dx.
\end{displaymath}
\end{corollary}

Next we note that we may replace the nonhomogeneous boundary condition~(\ref{eq:thm:main1-asymptotic-1}) by a distributional source term by extending the function $w=w_u$ in Theorem \ref{thm:main1-claim1} to all of $\R^N$ by even reflection. More precisely, we have the following. 

\begin{theorem}
  \label{intro-even-reflection}
   For every $u \in L^1_0(\R^N)$, there is a unique
   distributional solution $w \in \cAG(\R^{N+1})$ of the Poisson problem
  \begin{equation}
    	\label{eq:Poisson-roblem-distributional-sense}
		-\div(|t|\,\nabla
                w_u) = 2\, u \LNdo \quad \text{in $\R^{N+1}$}
    \end{equation}    
    which is even in the $t$-variable and satisfies
    $$
    \lim_{|t|\to \infty}w_u(x,t)=0\qquad\text{for every $x \in \R^N$.}
    $$
    This solution is given by (\ref{eq:poisson-formula-wu}) for $(x,t) \in \R^{N+1}_+$, so $w\big|_{\R^{N+1}_+}=w_u$ with $w_u$ as in Theorem~\ref{thm:main1}.
    
    Here, $\LNdo$ denotes the product measure of the $N$-dimensional
Lebesgue-measure $\cL^N$ and the Dirac-measure $\delta_0$ on $\R$ at $t=0$.
\end{theorem}

If one rewrites the weighted operator
\begin{displaymath}
 \mathcal{A}_0=- t \Bigl(\Delta_x - \frac{1}{t}\partial_t(t\partial_t\cdot)\Bigr),
\end{displaymath}
then one sees that the part $\frac{1}{t}\partial_t(t\partial_t\cdot)$ is the radial
part of the Laplacian in $\R^2$. It is therefore convenient to \emph{double the
  variable $t$} of the distributional solution
  $w_u$ of Poisson
  problem~\eqref{eq:Poisson-roblem-distributional-sense}, that is, to 
  consider the mapping
\begin{displaymath}
  (x,y) \mapsto W_u(x,y):= w_y(x,|y|) \qquad (x,y) \in \R^{N+2} = \R^N \times
  \R^2.
\end{displaymath}
As we shall see in Section~\ref{sec:equiv-form-poiss-1} below, this simple transformation allows us to reformulate Theorem~\ref{thm:main1} as follows. In order to simplify the notation, we set $\R^{N+2}_*:= \R^N \times (\R^2 \setminus \{(0,0)\})$.

\begin{theorem}
   \label{thm:main2}
   For every $u \in L^1_0(\R^N)$, there is a unique
   distributional solution
   \begin{displaymath}
     W_u \in \cAG(\R^{N+2}) \cap
     C^\infty(\R^{N+2}_*)
   \end{displaymath}
   of the Poisson problem
  \begin{equation}
    \label{eq:distributional-2-extension-thm2}
    -\Delta {w_{u}} = 2 \pi u\, \LNdoo
    \quad \text{in $\R^{N+2}$,}  
  \end{equation}
   satisfying
   \begin{displaymath}
       \lim_{|y|\to \infty}W_u(x,y)=0\qquad\text{for every $x \in \R^N$.}
   \end{displaymath}
    Here, $\delta_{(0,0)}$ is the Dirac-measure on $\R^2$ at $y=(0,0)$.
    Moreover, the following statements hold:
    \begin{enumerate}
    \item[(i)] $W_u$ satisfies the Poisson formula
     \begin{equation}
       \label{eq:poisson-formula-loglap-2-dim}
       W_u(x,y)=\frac{c_N}{2}\int_{\R^N}\frac{u(\tilde x)}{(|x-\tilde
         x|^2+|y|^2)^{N/2}}\,d \tilde x\qquad \text{for every $(x,y)\in \R^{N+2}_*$}
     \end{equation}
     \item[(ii)]
      \begin{equation}
        \label{u-limit-L-1-loc-2-dim}
        \lim_{|y|\to 0}\frac{W_u(\cdot,y)}{\ln
          t}=- u \qquad \text{in $L^1_{loc}(\R^N)$}
      \end{equation}
     \item[(iii)]
     \begin{equation}
       \label{eq:characterization-loglap-2-dim}
       \loglap u = 2(\ln 2-\gamma)u - 2\lim_{|y| \to 0}\Bigl(W_u(\cdot,y) + u \log
       |y|\Bigr) \quad \text{in the distributional sense in $\R^N$.}
     \end{equation}
     \item[(iv)] If, moreover, $u$ is Dini continuous at a point $x \in \R^N$, then we have 
      \begin{equation}
        \label{eq:characterization-loglap-pointwise-R-2}
              \loglap u(x) = 2\,(\ln 2 - \gamma)\,u(x)
                             -  2\lim_{|y| \to 0}\Bigl(W_u(x,y) + u(x) \ln t\Bigr) 
      \end{equation}
      as a pointwise limit in $\R$.
       \end{enumerate}
 \end{theorem}

Theorem~\ref{thm:main2} reveils that 
the logarithmic Laplacian $\loglap$ is intimately related with harmonic
functions $w_{u}$ on $\R^{N+2}_*$. As an application of Theorem~\ref{thm:main2}, we can establish a weak unique
continuation property for this operator. More precisely, we have the
following result.

\begin{theorem}\label{thm:main3}
  Let $u \in L^1_0(\R^N)$ be a function with the property that there is a
  nonempty open subset $\Omega \subset \R^N$ such that
  \begin{displaymath}
    u \equiv 0 \quad \text{in $\Omega$} \quad \quad \text{and} \quad
    \quad \loglap u \equiv 0 \quad 
    \text{in $\Omega$ in the distributional sense.}
  \end{displaymath}
  Then $u \equiv 0$ on $\R^N$. 
\end{theorem}

We wish to point out at this stage that we are able to deduce the extension problems associated with $\loglap$ only on a formal level from the Caffarelli-Silvestre extension problem (\ref{Ext-s}), (\ref{eq:caffarelli-silvestre-boundary-representiation}), while we need to use the integral representation (\ref{eq:log-laplace}) for a rigorous complete proof. However, the formal derivations, starting with the Caffarelli-Silvestre extension, were extremely helpful to identify the precise form of the extension problem for $\loglap$, and therefore we include these derivations in the present paper in Section~\ref{sec:formal-derivations} below.

We close this introduction with further comments on extension properties for nonlocal operators and their applications.

  The extension property of the fractional Laplacian $(-\Delta)^s$ was
  first known in the square root case $s=1/2$, and, in particular, for
  the square root $A^{1/2}$ of a quite general class of linear operators
  $A$ defined on Banach spaces (and, more generally, on locally convex
  spaces, cf~\cite{MR1850825} and there reference therein). Since the
  work~\cite{CS} by Caffarelli and Silvestre, the extension property of
  the fractional Laplacian was generalized in various direction: for
  linear second order partial differential operators (including
  Schr\"odinger operators), which can be realized as nonnegative,
  densely defined, and self-adjoint operator on $L^2(\Omega,d\eta)$,
  where $\Omega$ is an open subset of $\R^N$ and $d\eta$ a positive
  measure on $\Omega$ (see, e.g.,~\cite{MR2754080}), for heat operators
  (cf~\cite{MR3709888}), for (linear) sectorial operators $A$ defined on
  Hilbert spaces (see, for instance, \cite{MR3772192}) on Banach spaces
  (see, e.g., \cite{MR3056307,MR4151098}, and for nonlinear accretive
  operators $A$ in Hilbert spaces \cite{MR4026441}. Another important
  breakthrough in this field has been established by Kwa\'{s}nicki and
  Mucha \cite{MR3859452} (see also \cite{assing2019extension}) showing
  that the nonlocal operator $\varphi(-\Delta)$ also admits an extension
  property provided $\varphi : (0,\infty)\to [0,\infty]$ is a so-called
  complete Bernstein function. This result has been made available for
  more general diffusion operators in~\cite{assing2019extension}. The
  result \cite{MR3859452} has been refined and generalized for linear
  $m$-accretive operators defined on Banach space
  in~\cite{HauerLee}. Note, the power function $f(\lambda)=\lambda^s$
  for $\lambda\ge 0$ and $0<s<1$ is a complete Bernstein function. In
  particular, $f(\lambda)=\ln (\lambda+1)$ for $\lambda\ge 0$ is a
  complete Bernstein function, but $\ln \lambda$ for $\lambda>0$ is not
  a Bernstein function. Thus, the extension problem obtained for the
  logarithmic Laplacian $\loglap$ obtained in Theorem~\ref{thm:main1} is
  quite a surprise, and extends the existing literature in this field of
  research.\bigskip



  In the literature, one finds numerous applications of the extension property
  of the fractional Laplacian $(-\Delta)^s$. For example, it was applied to derive
  monotonicity formulas and regularity estimates for solutions and free boundary problems governed by the fractional Laplacian and the fractional perimeter (see, e.g., \cite{TVZ,CS0,CRS}), to establish $L^1$-$L^{\infty}$ regularity estimates
  of solutions to the fractional porous medium equation (see, e.g.,
  \cite{MR2737788,MR2954615}), to prove unique continuation principles (see, e.g.,
  \cite{FF,R-2014,MR4083776}) which were used, in \cite{MR4083776}, to solve the fractional Calder\'on problem,
  to derive spectral estimates for eigenvalue problems governed by the fractional Laplacian (see, e.g. \cite{FL,FLS})
  or to provide a classification of isolated
  singularities of solutions to a semilinar fractional Laplace equation
  on a punctured ball, where the nonlinearity has a power growth (see
  \cite{A0,CV2}).

  We hope that the extension property derived in this paper allows to extend some of these applications to the logarithmic Laplacian. Theorem~\ref{thm:main3} is a first example in this regard.\bigskip


 The organization of this paper is as follows. In Section~\ref{sec:first-derivation}, we provide two \emph{formal}
 derivations of the extension problem for the logarithmic Laplacian
 $\loglap$ starting from the Cafferalli-Silvestre extension (\ref{eq:caffarelli-silvestre-boundary-representiation}).
 Here we require additional hypothesis including regularity assumptions on $u$ and, on an abstract level, the existence of a uniform limit $v_0$ of $s$-dependent functions $v_s$ arising in the asymptotic expansion of the Cafferalli-Silvestre extension $w_s$ in (\ref{eq:caffarelli-silvestre-boundary-representiation}). Then, in Section~\ref{sec:equiv-form-poiss-1}, we derive the unique existence and integral representation of solutions of the degenerate Neumann problem in Theorem~\ref{thm:main1} and of the Poisson problems in Theorems~\ref{intro-even-reflection} and \ref{thm:main2}. In Section~\ref{sec:rigor-deriv-extens}, 
 we rigorously prove the characterizations (\ref{eq:characterization-loglap}) and (\ref{eq:characterization-loglap-pointwise}) {\em without} using the Cafferalli-Silvestre extension but using the integral representation (\ref{eq:log-laplace}) instead. Section~\ref{sec:an-application:-weak} is dedicated to the proof of weak unique contination property for $\loglap$ stated in Theorem~\ref{thm:main3}. Finally, in the Appendix, we provide the calculation of some integrals which we need in our proofs and which we could not find in the literature.


\setcounter{equation}{0}
\section{Formal derivations of the extension problem for the logarithmic Laplacian}
\label{sec:formal-derivations}

In this section we derive the representation formula (\ref{eq:characterization-loglap-pointwise}) on a formal level, starting from the Caffarelli-Silvestre extension problem for the fractional Laplacian.

Let $u \in C^\beta_c(\R^N)$ for some $\beta>0$, and let $s \in (0,\beta)$ in the following. Moreover, let $w_s:\R^{N+1}_+ \to \R$ denote the $s$-harmonic Caffarelli-Silvestre extension of $u$, which solves (\ref{Ext-s}) and gives rise to the representation 
  \begin{equation}
  \label{eq:caffarelli-silvestre-boundary-representiation-1}
  (-\Delta)^s u(x) = - d_s \lim_{t \to 0^+} t^{1-2s}\partial_t w_s(x,t) \qquad \text{for $x \in \R^N$}
\end{equation}
 with $d_s$ given by \eqref{c-s-extension-d-s-constant}, i.e., $d_s = 2^{2s-1} \frac{\Gamma(s)}{\Gamma(1-s)}$. By l'Hopital's rule, we may rewrite (\ref{eq:caffarelli-silvestre-boundary-representiation}) as 
  \begin{equation}
  \label{eq:caffarelli-silvestre-boundary-representiation-2}
[(-\Delta)^s u](x)= - 2s
d_s \lim_{t \to 0}\frac{w_s(x,t)-w_s(x,0)}{t^{2s}}=- 2s d_s \lim_{t \to
  0}\frac{w_s(x,t)-u(x)}{t^{2s}}\qquad \text{for $x \in \R^N$.}  
\end{equation}
Moreover, it is easy to see from the precise form of $d_s$ that
\begin{equation}
   \label{eq:d-s-asymptotics}
s d_s \to 1/2 \quad \text{and}\quad 
\frac{2s d_s-1}{s} = \frac{2^{2s} \frac{\Gamma(1+s)}{\Gamma(1-s)}-1}{s}
\to 2 \ln 2 + 2 \Gamma'(1)= 2(\ln 2 - \gamma) \quad \text{as
  $s \to 0^+$.}   
 \end{equation}
In order to characterize $\loglap u$ by an extension problem, it is useful to define $v_s:\R^{N+1}_+ \to \R$ for $s \in (0,\beta)$ by 
\begin{equation}
  \label{eq:def-v-s}
v_s(x,t)=\frac{2}{s} \Bigl(\frac{w_s(x,t)-u(x)}{t^{2s}}+ u(x)\Bigr). 
\end{equation}
With the help of (\ref{eq:caffarelli-silvestre-boundary-representiation-2}) and (\ref{eq:d-s-asymptotics}), we then compute that 
\begin{align}
  \loglap u (x)&= \lim_{s \to 0^+}\frac{(-\Delta)^s u(x) - u(x)}{s}=
   - \lim_{s \to 0^+} \lim_{t\to 0^+}\Bigl(2 d_s \frac{w_s(x,t)-u(x)}{t^{2s}} +\frac{u(x)}{s}\Bigr)\nonumber \\
               &=  - \lim_{s \to 0^+} \lim_{t\to 0^+}\Bigl(s d_s v_s(x,t) -
                 \frac{2s d_s -1}{s}u(x)\Bigr)\nonumber\\
  &=-\frac{1}{2}
   \lim_{s \to 0^+}\lim_{t\to 0^+}v_s(x,t)  + 2(\ln 2-\gamma) u(x).\label{first-char-loglap}
\end{align}
This is the starting point of our formal derivations of the extension problem for $\loglap u$, which are based on assuming the existence of 
\begin{equation}
  \label{eq:v0-limit}
v_0(x,t):= \lim_{s \to 0^+} v_s(x,t) \quad \text{as a locally uniform limit in $(x,t) \in \overline{\R^{N+1}_+},$}   
\end{equation}
Under this assumption, we may change the order of limits in (\ref{first-char-loglap}) and arrive at 
\begin{equation}
\label{first-char-loglap-1}
\loglap u(x) = \frac{1}{2} \lim_{t \to 0^+}v_0(x,t)  + 2 (\ln 2 - \gamma)u(x)  
\end{equation}
So it remains to characterize the limit $\lim_{t \to 0^+}v_0(x,t)$ via a suitable local extension problem, and this will be done in two different ways in the following subsections. Our first derivation only uses the extension problem (\ref{Ext-s}), (\ref{eq:caffarelli-silvestre-boundary-representiation}) in its differential form, while the second derivation will use a Poisson kernel representation of the Caffarelli-Silvestre extension.

\subsection{A first formal derivation of the
  extension problem}
\label{sec:first-derivation}

For a first formal derivation of an extension problem for $\loglap u$, we assume, for simplicity, throughout this section that 
$$
u \in C^2_c(\R^N)
$$ 
is fixed. The idea is to first derive a partial differential equation for the functions $v_s$ defined in (\ref{eq:def-v-s}) and then for $v_0$ by passing to the limit $s \to 0^+$. For this it is convenient to extend $w_s$ and $v_s$ as even functions in $t$ on all of $\R^{N+1}$. Then we can rewrite (\ref{eq:def-v-s}) as  
\begin{equation} \label{eq trans-1}
  w_s(x,t) = (1-|t|^{2s})\,u(x) +   \frac{s\, |t|^{2s}}{2}\, v_s(x,t)\qquad \text{for $(x,t) \in \R^{N+1}$.}
\end{equation}

We start with the following observation, which might be of independent interest.
\begin{lemma}
  \label{lemma:distibutional-sense-caffarelli-silvestre}
Let $0 <s <\frac{1}{2}$. Then 
$w_s: \R^{N+1} \to \R$ satisfies  
\begin{equation}
  \label{eq:distibutional-sense-caffarelli-silvestre}
-\div (|t|^{1-2s} \nabla w_s) = \frac{2}{d_s}[(-\Delta)^s
u] \LNdo
\quad  {\rm in}\ \, \R^{N+1} 
\end{equation}
in the distributional sense. In
other words, for all $\phi \in C^\infty_c(\R^{N+1})$, one has that
\begin{displaymath}
  \label{eq:distibutional-sense-caffarelli-silvestre-definition}
  \int_{\R^{N+1}} w_s \Bigl(-\div (|t|^{1-2s} \nabla \phi)\Bigr)\,dxdt  
  = \frac{2}{d_s}\int_{\R^N} [(-\Delta)^s u](x)\phi(x,0)\,dx.
\end{displaymath}
\end{lemma}

\begin{proof}
Let $\phi \in C^\infty_c(\R^{N+1})$. Then integrating by parts twice in the
equation \eqref{Ext-s} yields that
\begin{align*}
 \int_{\R^{N+1}}w_s\, \div (|t|^{1-2s} \nabla \phi)\,dx\,dt 
&= \lim_{\eps \to 0^+}\int_{\R^{N}\times\{t\,:\,|t|> \eps\}}w_s\, \div
  (|t|^{1-2s} \nabla \phi)\,dx\,dt \\
&= \lim_{\eps \to 0^+}\eps^{1-2s} \int_{\R^{N}\times\{t\,:\, |t|=
  \eps\}}
  w_s\, \partial_{\nu} \phi\,d\sigma(x,t),\\
&\hspace{1.5cm} -\lim_{\eps \to 0^+}\int_{\R^{N}\times\{t\,:\,|t|> \eps\}}
  |t|^{1-2s}\nabla w_s\nabla \phi\,dx\,dt\\  
&= \lim_{\eps \to 0^+}\eps^{1-2s} \int_{\R^{N}\times\{t\,:\, |t|=
  \eps\}}
  \bigl(w_s \partial_{\nu} \phi -  \phi \partial_{\nu} w_s\bigr)\,d\sigma(x,t), 
\end{align*}
where $\nu(x,t)= -\frac{t}{|t|}.$ For $0 <s <\frac{1}{2}$ we have, since
$w_s$ is continuous on $\R^N$ and $\phi \in C^\infty_c(\R^{N+1})$,
\begin{displaymath}
\lim_{\eps \to 0}\eps^{1-2s} \int_{|t|= \eps}w_s \partial_\nu \phi\,d\sigma(x,t)= 0.
\end{displaymath}
Moreover, it follows from
(\ref{eq:caffarelli-silvestre-boundary-representiation}) and the
evenness of the function $w_s$ in $t$ that
\begin{align*}
(-\Delta)^s u(x) &= d_s \lim_{t \to 0^+} t^{1-2s}\partial_\nu w_s(x,t)\\
&=\frac{d_s}{2}\Big[ 
-\lim_{t \to 0^+} t^{1-2s}\partial_t
w_s(x,t)\mathds{1}_{\R^{n+1}_{+}}(x,t)
-\lim_{t \to 0^-} |t|^{1-2s}\partial_t
w_s(x,|t|)\mathds{1}_{\R^{n+1}_{-}}(x,t)\Big]
\\
&= \frac{d_s}{2} \lim_{t \to 0} |t|^{1-2s}\partial_\nu w_s(x,|t|).
\end{align*}
Therefore, 
\begin{align*}
   \int_{\R^{N+1}}w_s \div (|t|^{1-2s} \nabla \phi)\,dxdt  
  &= - \lim_{\eps \to 0^+}\eps^{1-2s} \int_{|t|= \eps}\phi(x,t)  \partial_\nu w_s(x,t) \,dxdt \\
  &= -\lim_{\eps \to 0^+}\int_{|t|=   \eps}\phi(x,t) |t|^{1-2s} \partial_\nu w_s(x,t) \,d\sigma(x,t)
\\&=- \frac{2}{d_s} \int_{\R^N}\phi(x,0) (-\Delta)^s u(x)\,dx
\end{align*}
as claimed.
\end{proof}

Next, we show which equation $v_s$ solves in the sense of distributions. 

\begin{lemma}\label{lem:2-2}
  Let $0 <s <\frac{1}{2}$. Then the function
  $v_s: \R^{N+1} \to \R$ given by the asymptotic ansatz \eqref{eq trans-1} satisfies  
    \begin{equation}
        \label{eq:solved-by-vs}
        \begin{split}
          - \div (|t| \nabla v_s) &= 2\,s \frac{t}{|t|} \partial_t v_s
          +4\,(s\, v_s-2\,u)\LNdo\\
          &\hspace{2cm} + \frac{4}{s\,d_s}(-\Delta)^s u \LNdo
          +2\,|t|\, \frac{|t|^{-2s}-1}{s} \Delta_x u
        \end{split}
      \end{equation}
      in $\R^{N+1}$ in the distributional sense.
\end{lemma}

\begin{proof}
    Indeed, by \eqref{eq:distibutional-sense-caffarelli-silvestre}, and
    by applying the asymptotic ansatz \eqref{eq trans-1}, one sees that
    \allowdisplaybreaks
    \begin{align*}
      &-\frac{4}{s \,d_s}\int_{\R^N} (-\Delta)^s u(x)\phi(x,0)\,dx\\ 
      &\qquad = \frac{2}{s}\int_{\R^{N+1}}w_{s} \,\div(|t|^{1-2s} \nabla \phi)\,dxdt\\
      &\qquad = \int_{\R^{N+1}}\Bigl(2\,\frac{1-|t|^{2s}}{s}\,u(x) +    
        |t|^{2s}\, v_s(x,t)\Bigr)\, \div (|t|^{1-2s} \nabla
        \phi)\,dxdt\\
      &\qquad =
        -\int_{\R^{N+1}}|t|^{1-2s}\,\nabla\Bigl(2\,\frac{1-|t|^{2s}}{s}\,u(x) 
        \Bigr)\nabla \phi\,dxdt+\int_{\R^{N+1}}    
        |t|^{2s}\, v_s(x,t)\,\div(|t|^{1-2s}\nabla\phi)\,dxdt\\
       &\qquad =
        -\int_{\R^{N+1}}|t|^{1-2s}\,2\,\frac{1-|t|^{2s}}{s}\,\nabla_{x}u(x)\nabla_{x}\phi\,dxdt\\
       &\hspace{2cm}
         -\int_{\R^{N+1}}|t|^{1-2s}\,2\,u(x)\,\partial_{t}\Bigl(\frac{1-|t|^{2s}}{s}\Bigr)\partial_t\phi\,dxdt\\
      &\hspace{4cm} 
        +\int_{\R^{N+1}}    
        |t|^{2s}\, v_s(x,t)\,\div(|t|^{1-2s}\nabla\phi)\,dxdt\\    
      &\qquad =
        -\int_{\R^{N+1}}|t|^{1-2s}\,2\,\frac{1-|t|^{2s}}{s}\,\nabla_{x}u(x)\nabla_{x}\phi\,dxdt
         +4\,\int_{\R^{N+1}}u(x)\,\frac{t}{|t|}\,\partial_t\phi\,dxdt\\
      &\hspace{4cm} 
        +\int_{\R^{N+1}}    
        v_s(x,t)\,\Bigl(\div_x(|t|\nabla_x\phi)+
        (1-2s)\,\frac{t}{|t|}\,\partial_{t}\varphi 
        +|t|\partial_{tt}\varphi\Bigr)\,dxdt\\    
      &\qquad =
        -\int_{\R^{N+1}}|t|^{1-2s}\,2\,\frac{1-|t|^{2s}}{s}\,\nabla_{x}u(x)\nabla_{x}\phi\,dxdt
         +4\,\int_{\R^{N+1}}u(x)\,\frac{t}{|t|}\,\partial_t\phi\,dxdt\\
      &\hspace{2cm} + \int_{\R^{N+1}}  v_s(x,t)\Bigl(\div_x (|t| \nabla_{x}\phi)
        +\partial_t (|t| \partial_t \phi) -2s \frac{t}{|t|}\partial_t \phi \Bigr)\,dxdt\\
      &\qquad = -2\,\int_{\R^{N+1}} \Bigl(|t| \frac{|t|^{-2s}-1}{s}\nabla_x u(x)\nabla_x \phi(x,t)
        - 2 u(x) \frac{t}{|t|} \partial_t \phi(x,t)\Bigr)\,dxdt\\
      &\hspace{2cm} + \int_{\R^{N+1}}  v_s(x,t)\Bigl(\div (|t| \nabla
        \phi)-2s \frac{t}{|t|}\partial_t \phi \Bigr)\,dxdt
       \end{align*}
       for every $\varphi \in C^\infty_c(\R^{N+1})$. Since for
       $\varphi \in C^\infty_c(\R^{N+1})$,
      \begin{align*}
        \int_{\R^{N+1}}u(x)\, \frac{t}{|t|} \partial_t \phi(x,t)\,dxdt&=
        \int_{\R^{N}}u(x)\int_0^{\infty} \partial_t \phi(x,t)\,dtdx
        -\int_{\R^{N}}u(x)\int_{-\infty}^0 \partial_t \phi(x,t)\,dtdx\\
        &=-2 \int_{\R^{N}}u(x)\,\varphi(x,0)\,dx,
      \end{align*}
      it follows that
        \begin{align*}
        -\frac{4}{s\,d_s}\int_{\R^N} (-\Delta)^s u(x)\phi(x,0) dx
      & = 2\,\int_{\R^{N+1}}|t| \frac{|t|^{-2s}-1}{s}\Delta_x
        u(x)\,\phi(x,t) dx dt
        - 8 \int_{\R^N} u(x)\phi(x,0) dx\\
      &\hspace{2.2cm} + \int_{\R^{N+1}}  v_s\div (|t| \nabla \phi) dxdt 
        -2s \int_{\R^{N+1}}v_s \frac{t}{|t|}\partial_t \phi dxdt
      \end{align*}    
      for every $\varphi \in C^\infty_c(\R^{N+1})$. Finally, an
      integration by parts w.r.t. the variable $t$, shows that
      \begin{align*}
        \int_{\R^{N+1}}v_s \frac{t}{|t|}\partial_t \phi \,dxdt&=
        \int_{\R^{N}}\int_0^{\infty}v_s \partial_t \phi\,dtdx
        -\int_{\R^N}\int_{-\infty}^{0}v_s \partial_t \phi \,dtdx\\
        &= -2\int_{\R^{N}}v_s(x,0)\,\phi(x,0)\,dx-\int_{\R^{N}}\int_0^{\infty}\partial_tv_s\,
          \phi\,dtdx
          +\int_{\R^N}\int_{-\infty}^{0}\partial_tv_s \,\phi \,dxdt\\
        &= -2\int_{\R^{N}}v_s(x,0)\,\phi(x,0)\,dx-\int_{\R^{N+1}}\frac{t}{|t|}\partial_tv_s\,
          \phi\,dtdx
      \end{align*}
      for every $\varphi \in C^\infty_c(\R^{N+1})$ and so,
    \begin{equation}
      \label{eq:distrib-eq-of-vs}
      \begin{split}
        &-\frac{4}{s \,d_s}\int_{\R^N} (-\Delta)^s u(x)\,\phi(x,0)\,dx\\ 
      &\qquad = 2\,\int_{\R^{N+1}}|t| \frac{|t|^{-2s}-1}{s}
        \Delta_x u(x)\,\phi(x,t)\,dx dt
        - 8 \int_{\R^N} u(x)\phi(x,0)\,dx\\
      &\hspace{2cm} + \int_{\R^{N+1}}  v_s\div (|t| \nabla \phi)
        \,dxdt +4s \int_{\R^N}v_s(x,0)\phi(x,0)\,dx
        + 2s \int_{\R^{N+1}} \partial_t v_s \frac{t}{|t|} \phi \,dxdt   
     \end{split}
   \end{equation}
     for every $\varphi \in C^\infty_c(\R^{N+1})$, showing that $v_s$
     satisfies~\eqref{eq:solved-by-vs}. 
    \end{proof} 


By assumption~\eqref{eq:v0-limit} and since $\lim \limits_{s \to 0^+} (-\Delta)^s u = u$ on $\R^N$, $\lim \limits_{s \to 0^+} s d_s = 1/2$, 
we can pass to the limit in \eqref{eq:distrib-eq-of-vs} as
$s \to 0^+$. Then, we find that
\begin{align*}
        -8\int_{\R^N} u(x)\phi(x,0)\,dx 
      & = 2\,\int_{\R^{N+1}}|t| (-2\,\ln |t|)
        \Delta_x u(x)\,\phi(x,t)\,dx dt
        - 8 \int_{\R^N} u(x)\phi(x,0)\,dx\\
      &\hspace{4cm} + \int_{\R^{N+1}}  v_0\,\div (|t| \nabla \phi)
        \,dxdt  
\end{align*}
for every $\varphi \in C^\infty_c(\R^{N+1})$. Thereby we have shown
the following. 

\begin{lemma}
  Let $u\in C^2_c(\R^N)$, $v_s$ given by~\eqref{eq trans-1}, and
  $v_0\in C(\R^{N+1})$ be the locally uniform limit~\eqref{eq:v0-limit}
  of $\{v_s\}_{s\in (0,1/2)}$. Then, $v_0$ is a distributional solution
  of the equation
  \begin{equation}
    \label{eq:limit-s-0-formal}
    -\div (|t| \nabla v_0)  =  -4\,|t| \bigl(\ln |t|\bigr) \Delta_x
    u\qquad\text{in $\R^{N+1}$.}
  \end{equation}
\end{lemma}

Further, the following auxiliary result holds.

\begin{lemma}
  \label{lem:aux-ln}
 We have 
  \begin{equation}
    \label{eq:lem-diff-ln}
    \partial_t (|t|\partial_t \ln |t|)= 2\, \delta_{0} \qquad \text{in $\R$}
  \end{equation}
   in the distributional sense.
\end{lemma}

\begin{proof}
  Let $\varphi\in C^{\infty}_c(\R)$ and $b>0$ large enough such that
  $\supp(\varphi)\subseteq (-b,b)$. Then, integration by parts shows that
  \begin{align*}
    \int_{\R}&\ln\,|t|\,\partial_t (|t|\partial_t \varphi(t))\,dt
     =
      \lim_{\varepsilon\to0^+}\int_{\varepsilon}^{b}\ln\,t\,\partial_t
      (t\partial_t \varphi(t))\,dt+ \lim_{\varepsilon\to0^+}\int^{-\varepsilon}_{-b}\ln\,(-t)\,\partial_t
      ((-t)\partial_t \varphi(t))\,dt\\
    & =
      \lim_{\varepsilon\to0^+}\Big[t\,\ln\,t\,
      \partial_t\varphi(t)\Big]_{\varepsilon}^{b}
      - \int_{\varepsilon}^{b}\partial_t\varphi(t))\,dt + \lim_{\varepsilon\to0^+}\Big[
      (-t)\,\ln\,(-t)\,\partial_t\varphi(t)\Big]^{-\varepsilon}_{-b}
      +\int^{-\varepsilon}_{-b}\partial_t \varphi(t)\,dt\\
    & = \lim_{\varepsilon\to0^+}
      -(\varphi(b)-\varphi(\varepsilon))+(\varphi(-\varepsilon)-\varphi(-b))=
      2\,\varphi(0).
  \end{align*}
\end{proof}

From Lemma~\ref{lem:aux-ln} and \eqref{eq:limit-s-0-formal}, we can now
obtain the following result, which concludes the first formal derivation of the extension problem in the form given by (\ref{eq:Poisson-roblem-distributional-sense}) and (\ref{eq:characterization-loglap-pointwise}). 

\begin{lemma}
  \label{lem:construction-of-wu}
   Suppose the hypotheses of this section hold. Then, the function 
    \begin{equation}
\label{def-w-first-formal}
      {w}(x,t):= \tfrac{1}{4}v_0(x,t) -  u(x)\, \ln |t|
      \qquad\text{for every $(x,t)\in \R^{N+1}$ with $t\neq 0$,}
    \end{equation}
   is a solution of the Poisson equation
    \begin{equation}
      \label{eq:diff-of-wu}
      -\div (|t| \nabla {w}) = 2\, u \LNdo
      \qquad \text{in $\R^{N+1}$}
    \end{equation}
    in the distributional sense; i.e., 
\begin{displaymath}
  \int_{\R^{N+1}} w_u \bigl(-\div (|t| \nabla \phi)\bigr)\,dxdt  
  = 2\int_{\R^N}  u(x)\,\phi(x,0)\,dx  \qquad \text{for all $\phi \in
C^\infty_c(\R^{N+1})$.}
\end{displaymath}
Moreover, we have 
\begin{equation}
  \label{bounary trace}
\loglap u  =2(\ln 2 - \gamma)  u - 2\,\lim_{|t| \to 0}\Bigl(w(\cdot,t) +  u \ln |t|\Bigr) \qquad \text{on $\R^N$.}  
\end{equation}
  \end{lemma}

  \begin{proof}
 By \eqref{eq:limit-s-0-formal} and
    \eqref{eq:lem-diff-ln}, we have 
    \begin{align*}
      -\div (|t| \nabla {w}) 
      &= -\tfrac{1}{4}\div (|t| \nabla v_0)+ 
        \div (|t| \nabla (u\,\ln |t|))\\
      &=  - |t| \bigl(\ln |t|\bigr) \Delta_x
        u + |t|
        \bigl(\ln |t|\bigr) \Delta_x u + u\,\partial_t (|t| \partial_t \ln |t|)\\
      &=  2\,u \LNdo
    \end{align*}
    in the distributional sense in $\R^{N+1}$. Moreover, (\ref{bounary trace}) follows from (\ref{first-char-loglap-1}) and the definition of $w$ in (\ref{def-w-first-formal}). 
  \end{proof}
  
\subsection{A second formal derivation of the
 extension problem}
\label{sec:second-derivation}

Let again $u \in C^\beta_c(\R^N)$ for some $\beta>0$, and let $s \in (0,\beta)$ in the following. Similarly as in the last section, we consider the auxiliary function $v_s$ given by (\ref{eq:def-v-s}), and we assume (\ref{eq:v0-limit}), i.e., the existence of
$v_0$ as a locally uniform limit of the functions $v_s$ on $\overline {\R^{N+1}_+}$ as $s \to 0^+$.

The additional tool we shall use now is the Poisson kernel representation of the Caffarelli-Silvestre extension $w_s$. More precisely, as noted in \cite{Caffarelli-Silvestre} and \cite{CS}, we have 
\begin{equation}
  \label{eq:Poisson-rep-ws}
    w_s(x,t)=p_{N,s}\,t^{2s}\,\int_{\R^N}(|x-\tilde{x}|^2+t^2)^{-\frac{N+2s}{2}}
    u(\tilde{x})\,d\tilde{x} \qquad \text{for every $(x,t)\in \R^{N+1}_+$,}
  \end{equation}
  where the constant $p_{N,s}=\pi^{-N/s}s\frac{\Gamma(\frac{N}{2}+s)}{\Gamma(1+s)}$ is chosen such that
  $$
p_{N,s}\,t^{2s}\,\int_{\R^N}(|x-\tilde{x}|^2+t^2)^{-\frac{N+2s}{2}}
\,d\tilde{x} = p_{N,s} \int_{\R^N}(|z|^2+1)^{-\frac{N+2s}{2}}\,dz = 1 \qquad \text{for all $(x,t) \in \R^{N+1}_+$.} 
  $$
  By (\ref{eq:def-v-s}) and~\eqref{eq:Poisson-rep-ws}, we have 
 \begin{equation}
   \label{eq:2aux}
 \begin{split}
  \tfrac{1}{2} v_s(x,t)&=\frac{1}{s\,t^{2s}}\Bigl(w_s(x,t)-u(x)\Bigr)+\frac{1}{s}u(x)\\
   &= \frac{p_{N,s}}{s}\,\int_{\R^d}(|x-\tilde{x}|^2+t^2)^{-\frac{N+2s}{2}}
    u(\tilde{x})\,d\tilde{x}-\frac{t^{-2s}-1}{s}\,u(x)
 \end{split}
\end{equation}
for every $(x,t) \in \R^{N+1}_+$ and $0<s<1$. Using the explicit value of $p_{N,s}$ given above, it is easy to see $\frac{p_{N,s}}{s}\to c_N$
as $s\to0^+$, where $c_N$ is defined in (\ref{eq:def-c-n-rho-n}). Since $u$ is compactly supported by assumption, the limit
\begin{displaymath}
  \lim_{s\to0^{+}}\int_{\R^N}(|x-\tilde{x}|^2+t^2)^{-\frac{N+2s}{2}}
    u(\tilde{x})\,d\tilde{x}=\int_{\R^N}(|x-\tilde{x}|^2+t^2)^{-\frac{N}{2}}
    u(\tilde{x})\,d\tilde{x}
\end{displaymath}
exists by Lebesgue's dominated convergence theorem for every fixed $(x,t) \in \R^{N+1}_+$. Consequently, sending $s\to 0^+$
in \eqref{eq:2aux} yields that
 \begin{equation}
   \label{eq:v0-wu-pert-lnt}
   \tfrac{1}{2}v_0(x,t)=c_N\,\int_{\R^N}(|x-\tilde{x}|^2+t^2)^{-\frac{N}{2}}
    u(\tilde{x})\,d\tilde{x}+2 u(x) \ln t  \qquad \text{for every $(x,t)\in \R^{N+1}_+$.}
 \end{equation}

 \begin{remark}
   Let us briefly comment here why it appears difficult to show that
   (\ref{eq:v0-wu-pert-lnt}) arises from (\ref{eq:2aux}) as a {\em
     uniform} limit in bounded subsets of $\R^{N+1}_+$ as $s \to 0^+$,
   as assumed in (\ref{eq:v0-limit}). In fact neither one of the
   $s$-dependent functions
 $$
 (x,t) \mapsto
 \frac{p_{N,s}}{s}\,\int_{\R^d}(|x-\tilde{x}|^2+t^2)^{-\frac{N+2s}{2}}
 u(\tilde{x})\,d\tilde{x}, \qquad \quad \ (x,t) \mapsto
 \frac{t^{-2s}-1}{s}\,u(x)
 $$
 converges uniformly in bounded subsets of $\R^{N+1}_+$, so the
 uniformity of the convergence must be a result of cancellation
 effects. We recall here that the locally uniform convergence was needed
 to exchange the order of limits in the formula
 (\ref{first-char-loglap}) for the logarithmic Laplacian.
\end{remark}

 Once we have (\ref{eq:v0-wu-pert-lnt}), we can use (\ref{first-char-loglap-1}) to show that
\begin{equation}
  \label{bounary trace-second}
\loglap u  =2(\ln 2 - \gamma)  u - 2\,\lim_{t \to 0^+}\Bigl(w_u(\cdot,t) +  u \ln t\Bigr) \qquad \text{on $\R^N$.}  
\end{equation}
with the function $w_u$ given by (\ref{eq:poisson-formula-wu}), i.e. 
  \begin{equation*}
     w_u(x,t):= \frac{c_N}{2} \int_{\R^N} 
     \left(|x-\tilde x|^{2}+t^2\right)^{-\frac{N}{2}}u(\tilde x)\,d
     \tilde x\qquad
     \text{for every $(x,t)\in \R^{N+1}_+$.}
  \end{equation*}
  Hence, in order to complete a formal derivation of the extension property of the logarithmic laplacian in the form given by Theorem~\ref{thm:main1} for a function $u \in C_c^\beta(\R^N)$, it remains to show that
  $w_u$ satisfies the equation
  $$
  -\div(t\,\nabla
                w_u) = 0  \qquad \text{in $\R^{N+1}_+$}
  $$
  in classical sense together with the asymptotic boundary condition 
                $$
                \lim_{t \to 0^+} t \partial_t w_u(\cdot,t) = -u \qquad \text{in $L^1_{loc}(\R^N)$}
                $$
This will be done by direct computation under the more general assumption 
                $u \in L^1_0(\R^N)$ in Lemmas~\ref{equiv-form-poiss-lemm-prelim-1} and~\ref{equiv-form-poiss-lemm-prelim-2-2} below, which completes the second formal derivation of the extension property of $\loglap$.

 \section{Equivalent formulations of the Poisson problem associated with the logarithmic Laplacian}
\label{sec:equiv-form-poiss-1}

In this section we derive the unique existence of solutions of the degenerate Neumann problem in Theorem~\ref{thm:main1} and of the Poisson problems in Theorems~\ref{intro-even-reflection} and \ref{thm:main2}. Moreover, we shall show that these are essentially equivalent extension problems, and their solutions are given by the Poisson type integral formulas (\ref{eq:poisson-formula-wu}) and (\ref{eq:poisson-formula-loglap-2-dim}).

We first point out that, if $u \in L^1_0(\R^N)$, then $w_u(x,t)$ given by
\eqref{eq:poisson-formula-wu} is well-defined pointwisely for every $(x,t)\in \R^{N+1}_+$,
since
\begin{displaymath}
|w_u(x,t)| \le \int_{\R^N} \frac{|u(y)|}{(|x-y|^2+t^2)^{\frac{N}{2}}}\,dy \le C_t \int_{\R^N} \frac{|u(y)|}{(|y|+1)^{N}}\,dy
  < \infty
\end{displaymath}
with $C_t:= \sup_{y \in \R^N} \frac{(|y|+1)^N}{(|x-y|^2+t^2)^{\frac{N}{2}}} < \infty$. The following is the main result of this section. 

 \begin{theorem}
\label{main-theorem-equivalence}
   Let $u \in L^1_0(\R^N)$. Then, for a function $w: \R^{N+1}_+ \to \R$, the following properties are equivalent:
   \begin{enumerate}
   \item[(i)] $w= w_u$ with $w_u$ given by \eqref{eq:poisson-formula-wu}.
   \item[(ii)] $w \in \cAG(\R^{N+1}_+)$, and it satisfies
     \begin{equation}
       \label{eq:eq-w-u}
     -\div (t \nabla w(x,t)) = 0\quad \text{in $\R^{N+1}_+$}
     \end{equation}
     together with the asymptotic boundary conditions
     \begin{equation}
       \label{eq:asymptotic-w-1-equiv}
     -\lim_{t\to 0^+} t \partial_t w(\cdot,t) =u \qquad \text{ in $L^1_{loc}(\R^N)$}
     \end{equation}
     and
     \begin{equation}
       \label{eq:asymptotic-w-2-equiv}
      w(x,t) \to 0 \qquad \text{pointwisely as $t \to \infty$ for every $x \in \R^N$.} 
     \end{equation}   
   \item[(iii)] After extending $w$ to $\R^{N+1}$ by even reflection in
     the $t$-variable, we have $w \in \cAG(\R^{N+1})$, and it satisfies
     \begin{equation}
       \label{eq:eq-w-u-even}
-\div (|t| \nabla\,w(x,t)) = 2 u \LNdo \quad \text{in $\R^{N+1}$}
     \end{equation}
     in distributional sense together with the asymptotic boundary condition
     \begin{equation}
       \label{eq:asymptotic-w-2-even equiv}
      w(x,|t|) \to 0 \qquad \text{as $|t| \to \infty$  for every $x \in \R^N$.} 
     \end{equation}   
   \item[(iv)] We have $W \in \cAG(\R^{N+2})$ for the function $(x,y) \to W(x,y)= w(x,|y|)$, and it satisfies
    \begin{equation}
       \label{eq:eq-W-u}
     -\Delta W = 2 \pi u \LNdoo \quad \text{in $\R^{N+2}$}
\end{equation}   
     in distributional sense together with the asymptotic boundary condition
     \begin{equation}
       \label{eq:asymptotic-W-equiv}
      W(x,y) \to 0 \qquad \text{as $|y| \to \infty$ for every $x \in \R^N$.} 
     \end{equation}   
   \end{enumerate}
 \end{theorem}

 The remainder of this section is devoted to the proof of this theorem. We first recall that the space $L^1_0(\R^N)$ defined in (\ref{eq:def-L-1-0-space}) belongs to the more general family of weighted $L^1$-spaces
 \begin{equation}
   \label{eq:L-1-sigma}
L^1_\sigma(\R^N):= \Bigl\{ v \in L^1_{loc}(\R^N)\::\: 
 \|v\|_{L^1_\sigma}:= \int_{\R^N}\frac{|v(y)|}{(1+|y|)^{\sigma}}\,dy < \infty \} 
 \end{equation}
 with parameter $\sigma \ge 0$. We then note the following estimate in the norm $\|\cdot\|_{L^1_\sigma}$.

 \begin{lemma}
\label{equiv-form-poiss-lemm-prelim-1}   
Let $u \in L^1_0(\R^N)$, and let $w_u: \R^{N+1}_+ \to \R$ be given by \eqref{eq:poisson-formula-wu}.
Then $w_u(\cdot,t) \in L^1_\sigma(\R^N)$ for every $t>0$, $\sigma>0$, and there exists a constant $C_\sigma>0$ with
\begin{equation}
  \label{eq:C-sigma-est}
\|w_u(\cdot,t)\|_{L^1_\sigma} \le C_\sigma (1+ \ln^- t)\qquad \text{for every $t>0$.}
\end{equation}
Here, as usual, we put $\ln^{-} t = \max \{-\ln t,0\}$.
\end{lemma}

\begin{proof}
Let $\sigma>0$. In the following, the letter $C$ stands for positive constants which may depend on $\sigma$ but not on $t$. By definition and Fubini's theorem, we have that 
\begin{equation}
  \label{eq:j-formulation}
\|w_u(\cdot,t)\|_{L^1_\sigma} = \int_{\R^N} \frac{|w_u(t,x)|}{(1+|x|)^{N+\sigma}}\,dx  \le C \int_{\R^N} |u(y)| j_\sigma(y,t)dy    
\end{equation}
with
$$
j_\sigma(y,t)= \int_{\R^N}(1+|x|)^{-N-\sigma}(|x-y|^2 + t^2)^{-\frac{N}{2}}dx.
$$
We claim that
\begin{equation}
  \label{eq:j-estimate}
j_\sigma(y,t) \le (1+|y|)^{-N}(1+ \ln^- t).
\end{equation}
To see this, we first consider $|y| \le 3$. Then we have
\begin{align*}
j_\sigma(y,t)&= \int_{\R^N}(1+|z-y|)^{-N-\sigma}(|z|^2 + t^2)^{-\frac{N}{2}}dz \le
               C\int_{\R^N}(1+|z|)^{-N-\sigma}(|z|^2 + t^2)^{-\frac{N}{2}}dx\\
 &\le C \int_{B_1}(|z|^2 + t^2)^{-\frac{N}{2}}dx + C \int_{\R^N \setminus B_1}|z|^{-N-\sigma} (|z|^2 + t^2)^{-\frac{N}{2}}dx\\
             & \le C \min \{|\ln t|, t^{-N}\} \le C (1+|y|)^{-N}(1+ \ln^- t).
\end{align*}
Moreover, for $|y| \ge 3$, we have 
\begin{align*}
\int_{B_{\frac{|y|}{2}}}&(1+|x|)^{-N-\sigma}(|x-y|^2 + t^2)^{-\frac{N}{2}}dx \le  \Bigl(\bigl(\frac{|y|}{2}\bigr)^2 + t^2\Bigr)^{-\frac{N}{2}}  \int_{\R^N}(1+|x|)^{-N-\sigma}dx\\
                                                                            &\le C |y|^{-N} \le C (1+|y|)^{-N}(1+ \ln^- t),
\end{align*}
whereas
\begin{align*}
&\int_{B_{2|y|} \setminus B_{\frac{|y|}{2}}}(1+|x|)^{-N-\sigma}(|x-y|^2 + t^2)^{-\frac{N}{2}}dx \le C (1+|y|)^{-N-\sigma}
                                                                                                 \int_{B_{2|y|} \setminus B_{\frac{|y|}{2}}}(|x-y|^2 + t^2)^{-\frac{N}{2}}dx\\
                                                                                               &\le C (1+|y|)^{-N} \int_{B_{4|y|}}(|z|^2 + t^2)^{-\frac{N}{2}}dx \le C (1+|y|)^{-N} \int_{B_{\frac{4|y|}{t}}}(|z|^2 + 1)^{-\frac{N}{2}}dx \le  C (1+|y|)^{-N}(1+ \ln^- t)
\end{align*}
and also 
\begin{align*}
  \int_{\R^N \setminus B_{2|y|}}&(1+|x|)^{-N-\sigma}(|x-y|^2 + t^2)^{-\frac{N}{2}}dx \le C (|y|^2 + t^2)^{-\frac{N}{2}} \int_{\R^N}(1+|x|)^{-N-\sigma}dx \\
  &\le C |y|^{-N} \le C (1+|y|)^{-N}(1+ \ln^- t).
\end{align*}
Combining these estimates, we get (\ref{eq:j-estimate}). Using (\ref{eq:j-estimate}) in (\ref{eq:j-formulation}), we get 
\begin{equation}
  \label{eq:w-u-ln-est}
\|w_u(\cdot,t)\|_{L^1_{\sigma}(\R^N)} \le C (1+ \ln^- t) \quad \text{for all $t>0$,}
\end{equation}
as claimed.
\end{proof}

\begin{remark}
\label{counterexample}
If $u \in L^1_0(\R^N)$, then the functions $w_u(\cdot,t)$, $t >0$ do not need to belong to $L^1_0(\R^N)$ in general. A counterexample is given by
  the function
  \begin{displaymath}
    u(x):= \ln^{-\tau} (1+|x|)\qquad\text{for every $x\in \R^N$, and fixed
  $\tau \in (1,2)$.}
\end{displaymath}
This function $u$ belongs to $L^1_0(\R^N)$, since
  \begin{displaymath}
    \int_{\R^N} \frac{\ln^{-\tau} (1+|x|)}{(1+|x|)^N}dx 
    \le c_1 + c_2 \int_{1}^\infty \frac{1}{r \ln^{\tau} r}dr <\infty
\end{displaymath}
form some constants $c_1$, $c_2>0$. On the other hand, by Fubini's theorem we have
\begin{displaymath}
\int_{\R^N} \frac{w_u(x,t)}{(1+|x|)^N}\,dx = \frac{c_N}{2} \int_{\R^N}\ln^{-\tau} (1+|y|)\rho(y)\,dy,
\end{displaymath}
where
\begin{displaymath}
  \rho(y):= \int_{\R^N} \left(|x-y|^{2}+t^2\right)^{-\frac{N}{2}}\frac{1}{(1+|x|)^N}\,dx.
\end{displaymath}
Note that 
\begin{align*}
\rho(y) &\ge \int_{B_{|y|}}
          \left(|x-y|^{2}+t^2\right)^{-\frac{N}{2}}\frac{1}{(1+|x|)^N}\,dx \ge \left(\left(2|y|\right)^{2}+t^2\right)^{-\frac{N}{2}}
  \int_{B_{|y|}} \frac{1}{(1+|x|)^N}\,dx\\  
&\ge C_t\, |y|^{-N}\int_{1}^{|y|}
  \frac{r^{N-1}}{(1+r)^N}\,dr
 \ge C_t\, |y|^{-N}\ln |y|.
\end{align*}
Here, the letter $C_t$ stands for positive constants depending on $t$. Since $1-\tau>-1$, we conclude
that  
\begin{align*}
  \int_{\R^N} \frac{w_u(x,t)}{(1+|x|)^N}\,dx &
  \ge C_t\, \int_{\R^N \setminus B_3 }\ln^{-\tau} (1+|y|)|y|^{-N}\ln |y|
                                               \,dy\\
   &\ge C_t \int_{\R^N \setminus B_3 } |y|^{-N} \ln^{1-\tau}|y|\,dy= \infty.
\end{align*}
\end{remark}

Next, we show that the Poisson integral formula \eqref{eq:poisson-formula-wu} yields a classical solution of (\ref{eq:eq-w-u}).

\begin{lemma}
\label{equiv-form-poiss-lemm-prelim-2}   
Let $u \in L^1_0(\R^N)$, and let $w_u: \R^{N+1}_+ \to \R$ be given by \eqref{eq:poisson-formula-wu}. Then 
$w_u$ is a classical solution of (\ref{eq:eq-w-u}) in $\R^{N+1}_+$.
\end{lemma}

\begin{proof}
It is convenient to define $H_0 : \R^{N+1}\setminus\{0\}\to\R$ by
\begin{equation}
  \label{eq:H0-definition}
  H_{0}(X):=\frac{c_N}{2}|X|^{-N}\qquad\text{for every $X=(x,t)\in
    \R^{N+1}\setminus\{0\}$,}
\end{equation}
where $|X|=\sqrt{|x|^2+|t|^2}$ is the Euclidean norm on
$\R^{N+1}$. Setting $R:=|X|$ in the following, we directly compute, for $t>0$,
\begin{equation}
  \label{eq:0-harmonic-H0}
  \begin{split}
    \frac{1}{t}\div(t\,\nabla H_0)
    &=\Delta_XH_0+\frac{1}{t}\partial_t H_0\\[2mm]
    &= \frac{C_N}{2}\Big[(R^{-N})''+\frac{N}{R}(R^{-N})'+\frac1t \big(
    R^{-N-2} (-N) t\big) \Big]
    \\[2mm]
    &=\frac{C_N}{2}\Big[N R^{-N-2} -N R^{-N-2}\Big]=0.
  \end{split}
\end{equation}
Moreover, if $u \in L^1_0(\R^N)$ and $w_u: \R^{N+1}_+ \to \R$ is given by \eqref{eq:poisson-formula-wu}, then one easily sees that
\begin{displaymath}
  w_u(x,t)=(H_0(\cdot,t)\ast u)(x)=\int_{\R^N}H_0(x-\tilde{x},t)u(\tilde{x})\,d\tilde{x}
\end{displaymath}
for every $(x,t)\in \R^{N+1}$ with $t\neq 0$. Further, note that
\begin{align*}
  \left|\frac{\partial H_0}{\partial x_i}(x,t)\right|&\le N\,\max\{1,\tfrac{1}{|t|^2}\}\,|H_0(x,t)|,\\
  \left|\frac{\partial^2 H_0}{\partial x_i^2}(x,t)\right|
     &\le N\,\tfrac{1}{|t|^2}\Big[1+2(\frac{N}{2}+1)\Big]\,|H_0(x,t)|,\\
  \left|\frac{\partial H_0}{\partial t}(x,t)\right| &\le N\,\tfrac{1}{|t|}\,|H_0(x,t)|, 
\end{align*}
for every $X=(x,t)\in \R^{N+1}_{+}$. Thus, Lebesgue's dominated
convergence theorem yields that we may interchange integral and
differentiation and so, we can conclude that 
\begin{align*}
 \frac{1}{t}\div(t\,\nabla w_u(x,t))  
  &= \Delta_X  w_u(X)+\frac{1}{t}\partial_t w_u(X)\\ 
 & =\int_{\R^N} \Big( \Delta_x H_0(x-\tilde x,t) + 
\partial_t^2 H_0 (x-\tilde x,t)  + \frac{1}{t}\partial_t 
H_0(x-\tilde x,t) \Big) u(\tilde x) d\tilde x
 =0
\end{align*}
for every $X=(x,t)\in \R^{N+1}_{+}$. This shows that $w_u$ satisfies
(\ref{eq:eq-w-u}).
\end{proof}

Next, we show that $w_u$ defined by \eqref{eq:poisson-formula-wu}
satisfies the inhomogeneous Neumann condition
(\ref{eq:asymptotic-w-1-equiv}).

\begin{lemma}
\label{equiv-form-poiss-lemm-prelim-2-2}   
Let $u \in L^1_0(\R^N)$, and let $w_u: \R^{N+1}_+ \to \R$ be given by \eqref{eq:poisson-formula-wu}. Then 
$w_u$ satisfies the asymptotic boundary condition (\ref{eq:asymptotic-w-1-equiv}).
\end{lemma}

\begin{proof}
For $t>0$, a direct computation
 shows 
$$
  - t \partial_tw_u(x,t)=   \frac{Nc_{N}}{2} t^2  \int_{\R^N} \left(|x-\tilde
    x|^{2}+t^2\right)^{-\frac{N+2}{2}} u(\tilde x) \,d \tilde x =   \frac{N c_{N}}{2} \int_{\R^N} \left(|z|^{2}+1\right)^{-\frac{N+2}{2}} u(x+t z) \,d z
$$
Moreover, we have
$$
  \frac{Nc_{N}}{2}\Bigl(  \int_{\R^N}(|z|^2+1)^{-\frac{N+2}{2}}\,dz\Bigr)= \frac{N  c_{N}}{2}\Bigl( \frac{\omega_N}{2}B(\frac{N}{2},1)\Bigr)=\frac{N}{2} B(\frac{N}{2},1)= \frac{N}{2}
\frac{\Gamma(\frac{N}{2})\Gamma(1)}{\Gamma(\frac{N}{2}+1)}=1,
$$
where $B$ refers to the Beta function. This allows us to write, for $t \le 1$, 
\begin{align*}
\frac{2}{Nc_{N}}&\Bigl(  - t \partial_tw_u(x,t)-u(x)\Bigr) 
= \int_{\R^N} \left(|z|^{2}+1\right)^{-\frac{N+2}{2}}\Bigl( u(x+t z)-u(x)\Bigr) \,d z \\ 
  &= \int_{B_{\frac{1}{\sqrt{t}}}}\dots dz  + \int_{B_{\frac{1}{t}} \setminus B_{\frac{1}{\sqrt{t}}}} \dots dz + \int_{\R^N \setminus B_{\frac{1}{t}}} \dots dz= I_1 (x,t)+I_2(x,t)+I_3(x,t).
\end{align*}
Consequently, (\ref{eq:asymptotic-w-1-equiv}) follows once we have shown, for given $R>0$, that 
  \begin{equation}
  \label{eq:sufficientj=123}
 \lim_{t \to 0^+} \|I_j(\cdot,t)\|_{L^1(B_R)} =0 \qquad \text{for $j=1,2,3$.}
\end{equation}
In the following, the letter $C$ denotes a positive constant, which
might change from line to line, depending on $R$, $N$ and $\alpha$, but {\em not} on $t$. We first note that 
\begin{align*}
  &\|I_1(\cdot,t)\|_{L^1(B_R)} \le \int_{B_{\frac{1}{\sqrt{t}}}} \left(|z|^{2}+1\right)^{-\frac{N+2}{2}}\|u(\cdot +t z)-u\|_{L^1(B_R)} \,d z \\
&\le \sup_{|{\tilde z}| \le \sqrt{t}}\|u(\cdot +{\tilde z})-u\|_{L^1(B_R)}\int_{\R^N} \left(|z|^{2}+1\right)^{-\frac{N+2}{2}}\,d z \le C \sup_{|{\tilde z}| \le \sqrt{t}}\|u(\cdot +{\tilde z})-u\|_{L^1(B_R)}.  
\end{align*}
Moreover, using the density of $C^1_c(\R^N)$ in $L^1_0(\R^N)$, it is easy to see that 
\begin{equation*}
  \lim_{t \to 0} \sup_{|{\tilde z}| \le \sqrt{t}}\|u(\cdot +
  {\tilde z})-u\|_{L^1(B_R)}
  = \lim_{{\tilde z} \to 0} \|u(\cdot + \tilde z)-u\|_{L^1(B_R)}= 0
\end{equation*}
and therefore 
$$
\|I_1(\cdot,t)\|_{L^1(B_R)} \to 0 \qquad \text{as $t \to 0^+$.}
$$
Moreover, we estimate
\begin{align*}
  \|I_2(\cdot,t)\|_{L^1(B_R)}  &\le \int_{B_{\frac{1}{t}} \setminus B_{\frac{1}{\sqrt{t}}}} \left(|z|^{2}+1\right)^{-\frac{N+2}{2}}\|u(\cdot +t z)-u\|_{L^1(B_R)} dz\\
&\le \sup_{|\tilde z| \le 1}\|u(\cdot +\tilde z)-u\|_{L^1(B_R)} \int_{B_{\frac{1}{t}} \setminus B_{\frac{1}{\sqrt{t}}}} \left(|z|^{2}+1\right)^{-\frac{N+2}{2}}\,d z \\
&\le 2 \|u\|_{L^1(B_{R+1})} \int_{\R^{N} \setminus B_{\frac{1}{\sqrt{t}}}} \left(|z|^{2}+1\right)^{-\frac{N+2}{2}}\,d z \\  
&\le C \|u\|_{L^1_0} \int_{\R^{N} \setminus B_{\frac{1}{\sqrt{t}}}} \left(|z|^{2}+1\right)^{-\frac{N+2}{2}}\,d z \to 0 \qquad \text{as $t \to 0$.}
\end{align*}
Finally, we estimate, for $|x| \le R$, 
\begin{align*}
  |I_3(x,t)| &\le \int_{\R^N \setminus B_{\frac{1}{t}}}|z|^{-(N+2)}|u(x+t z)-u(x)|dz\\
& \le 
\int_{\R^N \setminus B_{\frac{1}{t}}}|z|^{-(N+2)}|u(x+t z)|dz
               + |u(x)|\int_{\R^N \setminus B_{\frac{1}{t}}} |z|^{-(N+2)}\,d z \\
&= t^2 \int_{\R^N \setminus B_{1}(x)}|{\tilde x}-x|^{-(N+2)}|u({\tilde x})|d{\tilde x} + C t^2|u(x)|\\ 
             &\le C t^2 \Bigl( \int_{\R^N} (1+|{\tilde x}|)^{-N}|u({\tilde x})|d{\tilde x} +|u(x)|\Bigr)= C t^2 \bigl( \|u\|_{L^1_0} +|u(x)|\bigr).
\end{align*}
Here we used the fact that
$$
\sup_{|x| \le R,\: {\tilde x} \in \R^{N} \setminus B_1(x)} (1+|{\tilde x}|)^{N} |{\tilde x}-x|^{-(N+2)} < \infty.
$$
We thus conclude that
\begin{equation}
  \label{eq:I_3-est}
  \|I_3(\cdot,t)\|_{L^1(B_R)} \le C t^2 \Bigl( \|u\|_{L^1_0}|B_R|   +\|u\|_{L^1(B_R)}\Bigr)\le C t^2 \|u\|_{L^1_0} 
\to 0 \quad  \text{as $t \to 0^+$.} 
\end{equation}
Combining the estimates for $I_1,I_2$ and $I_3$, we have shown (\ref{eq:sufficientj=123}), and this finishes the proof of (\ref{eq:asymptotic-w-1-equiv}).
\end{proof}

The following lemma will be useful to prove the implication (ii) $\Lra$ (iii) in Theorem~\ref{main-theorem-equivalence}. 

\begin{lemma}
\label{equiv-form-poiss-lemm-prelim-3}   
Let $u \in L^1_{loc}(\R^N)$, and let $w \in \cAG(\R^{N+1}_+) \cap C^1(\R^{N+1}_+)$ satisfy
     \begin{equation}
       \label{eq:asymptotic-w-1-equiv-lemma}
      t \partial_t w(\cdot,t) \to -u \qquad \text{as $t \to 0^+$ in $L^1_{loc}(\R^N)$}
     \end{equation}
     Then we also have
          \begin{equation}
       \label{eq:asymptotic-w-1-equiv-variant}
      \frac{w(\cdot,t)}{\ln t} \to -u \qquad \text{in $L^1_{loc}(\R^N)$}
     \end{equation}
   \end{lemma}
     
     \begin{proof}
Let $R>0$. Then l'H\^opital's rule gives 
\begin{align*}
  \lim_{t \to 0^+}&\frac{1}{\ln t} \int_{B_R} \sqrt{1 + (w(x,t)+ u(x) \ln t)^2}\,dx =\lim_{t \to 0^+} t \frac{d}{dt} \int_{B_R} \sqrt{1 + (w(x,t)+ u(x) \ln t)^2}\,dx\\
  &=\lim_{t \to 0^+}   \int_{B_R}  \frac{(w(x,t)+ u(x) \ln t) (t \partial_t w(x,t) + u(x)) }{\sqrt{1 + (w(x,t)+ u(x) \ln t)^2}}\,dx
\end{align*}
while
\begin{align*}
  \Bigl| \int_{B_R}  \frac{(w(x,t)+ u(x) \ln t) (t \partial_t w(x,t) + u(x)) }{\sqrt{1 + (w(x,t)+ u(x) \ln t)^2}}\,dx\Bigr|& \le \int_{B_R}|t \partial_t w(x,t) + u(x)|\,dx\\
 &= \|t \partial_t w(\cdot,t)+ u\|_{L^1(B_R)} \to 0\quad \text{as $t \to 0^+$}
\end{align*}
by (\ref{eq:asymptotic-w-1-equiv}). Hence we find that 
$$
  \lim_{t \to 0^+}\frac{1}{|\ln t|} \int_{B_R} \sqrt{1 + (w(x,t)+ u(x) \ln t)^2}\,dx =  \lim_{t \to 0^+}\frac{1}{\ln t} \int_{B_R} \sqrt{1 + (w(x,t)+ u(x) \ln t)^2}\,dx = 0 
  $$
  and therefore also 
  $$
\lim_{t \to 0} \Bigl\|\frac{w(\cdot,t)}{\ln t} + u \Bigr\|_{L^1(B_R)} = \lim_{t \to 0^+}\frac{1}{|\ln t|} \int_{B_R} |w(x,t)+ u(x) \ln t|\,dx = 0.
$$
Since $R>0$ was chosen arbitrarily, this yields (\ref{eq:asymptotic-w-1-equiv-variant}).
\end{proof}

We may now complete the

\begin{proof}[Proof of Theorem~\ref{main-theorem-equivalence}]
  We first show that (i) implies (ii). By Lemma~\ref{equiv-form-poiss-lemm-prelim-2}, the function $w:= w_u$ satisfies (\ref{eq:eq-w-u}) in $\R^{N+1}_+$. Moreover, by the assumption $u \in L^1_0(\R^N)$ and the dominated convergence theorem, we have
$$
|w_u(x,t)| \le \int_{\R^N}\frac{|u({\tilde x})|}{(|x-{\tilde x}|^2 + t^2)^\frac{N}{2}} d\tilde x \to 0 \qquad \text{as $t \to \infty$ for every $x \in \R^N$,}
$$
which gives (\ref{eq:asymptotic-w-2-equiv}).
Moreover, from Lemma~\ref{equiv-form-poiss-lemm-prelim-1} with $\sigma=1$, we get the estimate
$$
\int_{\R^N} \frac{|w({\tilde x},t)|}{(1+|{\tilde x}|)^{N+1}}d{\tilde x} \le C_1 (1+ \ln ^- t) \qquad \text{for $t>0$ with a constant $C_1>0$.}
$$
From this we deduce that for every $X=(x,t) \in \R^{N+1}_+$ we have, with a constant $C_2>0$,
\begin{align*}
  &\|w\|_{L^1(B_1(X) \cap \R^{N+1}_+)} \le C_2(1+|x|)^{\sigma} \int_{\max\{t-1,0\}}^{t+1}\int_{\R^N} \frac{|w_u({\tilde x},\tau)|}{(1+|{\tilde x}|)^\sigma}d{\tilde x} d\tau\\
 &\quad \le C_1 C_2 (1+|x|)^{\sigma} \int_{\max\{t-1,0\}}^{t+1}(1+ \ln^- \tau)\,d\tau \le 3 C_1 C_2 (1+|x|)^{\sigma} \le
3 C_1 C_2 (1+|X|)^{\sigma}.
\end{align*}
Hence $w_u \in \cAG(\R^{N+1}_+)$. Finally, it follows from
Lemma~\ref{equiv-form-poiss-lemm-prelim-2-2} that $w_u$ satisfies the
asymptotic boundary condition~(\ref{eq:asymptotic-w-1-equiv}).

Next we show that (ii) implies (iii). Since $w \in \cAG(\R^{N+1}_+)$, we
have $w \in \cAG(\R^{N+1} \setminus (\R^N \times \{0\})$ after even
reflection of the $t$-variable, and then a covering argument shows that
$w \in \cAG(\R^{N+1})$. Moreover, (\ref{eq:asymptotic-w-2-even equiv})
follows directly from (\ref{eq:asymptotic-w-2-equiv}). Finally, it
follows from (\ref{eq:eq-w-u}) that
\begin{equation}
  \label{eq:w-u-even-1}
-\div (|t| w(x,t)) = 0 \quad \text{in $\R^{N+1}  \setminus (\R^N \times \{0\})$.} 
\end{equation}
To see that (\ref{eq:eq-w-u-even}) holds in distributional sense, we let $\varphi \in C^\infty_c(\R^{N+1}_+)$ and
$\eps>0$. Integrating by parts twice and using (\ref{eq:w-u-even-1}), we then find that
\begin{align}
 \int_{\R^{N+1}}w\, \div (|t| \nabla \phi)\,dx\,dt 
&= \lim_{\eps \to 0^+}\int_{\R^{N}\times\{t\,:\,|t|> \eps\}}w\, \div
  (|t| \nabla \phi)\,dx\,dt \nonumber \\
&= \lim_{\eps \to 0^+}\eps \int_{\R^{N}\times\{t\,:\, |t|=
  \eps\}}
  w\, \partial_{\nu} \phi\,d\sigma(x,t) \nonumber \\
&\hspace{1.5cm} -\lim_{\eps \to 0^+}\int_{\R^{N}\times\{t\,:\,|t|> \eps\}}
  |t|^{1-2s}\nabla w\nabla \phi\,dx\,dt\\  
&= \lim_{\eps \to 0^+}\eps \int_{\R^{N}\times\{t\,:\, |t|=
  \eps\}}
  \bigl(w \partial_{\nu} \phi -  \phi \partial_{\nu} w\bigr)\,d\sigma(x,t),\label{twice-integration-distributional} 
\end{align}
where $\nu(x,t)= -\frac{t}{|t|}.$ Next we recall that, by Lemma~\ref{equiv-form-poiss-lemm-prelim-3}, the asymptotic property (\ref{eq:asymptotic-w-1-equiv}) implies (\ref{eq:asymptotic-w-1-equiv-variant}) and therefore
$$
t \|w(\cdot,t)\|_{L^1(B_R)} \to 0 \qquad \text{as $t \to 0$}
$$
Using this together with the fact that $\partial_{\nu}\phi$ has compact support, we see that 
\begin{equation}
\label{zero-limit}  
\lim_{\eps \to 0}\eps \int_{|t|= \eps}w \partial_\nu \phi\,d\sigma(x,t)= 0.
\end{equation}
Moreover, (\ref{eq:asymptotic-w-1-equiv}) and the evenness of $w$ in $t$ imply directly that 
\begin{align}
\lim_{\eps \to 0^+}\eps \int_{\R^{N}\times\{t\,:\, |t|=
                     \eps\}} \phi \partial_{\nu} w \,d\sigma(x,t)
                   &= \lim_{\eps \to 0^+} \int_{\R^{N}} (\phi(x,\eps)+\phi(x,-\eps)) \eps \partial_{\nu} w(x,\eps) \,dx \nonumber\\
  &= - 2 \int_{\R^N} \phi(x,0)u(x)\,dx \label{remainder-limit}  
\end{align}
Combining (\ref{twice-integration-distributional}), (\ref{zero-limit}) and (\ref{remainder-limit}) gives
\begin{equation}
  \label{eq:combination-even-reflection}
 \int_{\R^{N+1}}w\, \div (|t| \nabla \phi)\,dx\,dt = - 2 \int_{\R^N} \phi(x,0)u(x)\,dx,
\end{equation}
and this yields (\ref{eq:eq-w-u-even}) in distributional sense.\\

Next we show the implication (iii) $\Longrightarrow (iv)$. Since
(\ref{eq:asymptotic-W-equiv}) is an obvious consequence of
(\ref{eq:asymptotic-w-2-even equiv}) and $W \in \cAG(\R^{N+2})$ follows
from $w \in \cA(\R^{N+1})$, we only have to show (\ref{eq:eq-W-u}) in
distributional sense.  Since $W$ is radially symmetric in the variable
$y \in \R^2$, it suffices to show that
\begin{equation}
  \label{eq:combination-harmonic}
 \int_{\R^{N+2}}W\, \Delta \Phi \,dx\,dy = - 2 \pi \int_{\R^N} \Phi(x,0)u(x)\,dx
\end{equation}
for functions $\Phi \in C^\infty_c(\R^N)$ which are radially symmetric
in the $y$-variable, i.e. for functions of the form
$\Phi(x,y) = \phi(x,|y|)$ with a function
$\phi \in C^\infty_c(\R^{N+1})$ which is even in the $(N+1)$-th
variable.  In this case, using cylinder coordinates and
(\ref{eq:combination-even-reflection}) gives
\begin{align*}
  \int_{\R^{N+2}}W\, \Delta \phi \,dx\,dy &= (2 \pi) \int_{\R^{N} \times (0,\infty)}t w(x,t)\Bigl(\frac{1}{t}\div (t \nabla)\phi(x,t)\Bigr)dx dt \\
&= \pi  \int_{\R^{N+1}}w(x,t)\div (|t| \nabla \phi)dx dt = - 2 \pi \int_{\R^N} \phi(x,0)u(x)\,dx,  
\end{align*}
and thus (\ref{eq:combination-harmonic}) holds.\\
Finally, to show that (iv) implies (i), it suffices to show that
(\ref{eq:eq-W-u}) has a unique solution $W$ in $\cAG(\R^{N+2})$
satisfying (\ref{eq:asymptotic-W-equiv}). To see this, let
$W_1, W_2 \in \cAG(\R^{N+2})$ be two such solutions, and let
$W:= W_1-W_2$. Then $W$ satisfies $\Delta W =0$ in distributional sense
on $\R^{N+2}$ and therefore also in classical sense by Weyl's
lemma. Moreover, since $W \in \cAG(\R^{N+2})$, there exists constants
$C,\sigma>0$ with
$$
\int_{B_1(X)} |W(X')|d X' \le  C (1+|X|)^\sigma \qquad \text{for all $X \in \R^{N+2}$.}
$$
Using the mean value property for harmonic functions, it then follows
that, after making $C$ larger if necessary,
$$
|W(X)| \le C(1+|X|)^{\sigma} \qquad \text{for all $X \in \R^{N+2}$.}
$$
Then, a classical Liouville property for harmonic functions implies that
$W$ is a polynomial. Consequently, for every fixed $x \in \R^N$, the
function $W(x,\cdot)$ is a polynomial in $y$ satisfying
$$
W(x,y) \to 0 \qquad \text{as $|y| \to \infty$}
$$
since (\ref{eq:asymptotic-W-equiv}) holds for both $W_1$ and $W_2$. Consequently, $W(x,\cdot) \equiv 0$ for every $x \in \R^N$ and therefore $W  \equiv 0$, which shows $W_1 = W_2$. This gives the desired uniqueness result and finishes the proof of Theorem~\ref{main-theorem-equivalence}.
\end{proof}

\section{Rigorous derivation of the extension representation for $\loglap u$}
\label{sec:rigor-deriv-extens}

In this section, we prove the representation formula (\ref{eq:characterization-loglap}) for a fixed given function $u \in L^1_0(\R^N)$ with $w= w_u: \R^{N+1}_+ \to \R$ defined by the Poisson integral formula (\ref{eq:poisson-formula-wu}). Moreover, we prove the pointwise representation formula (\ref{eq:characterization-loglap-pointwise}) under the additional assumption that $u$ is Dini continuous at a point $x \in \R^N$. In combination with Theorem~\ref{main-theorem-equivalence} and Lemma~\ref{equiv-form-poiss-lemm-prelim-3}, this completes the proof of our main Theorems~\ref{thm:main1},~\ref{intro-even-reflection} and~\ref{thm:main2}.

We start with the pointwise representation formula (\ref{eq:characterization-loglap-pointwise}).

 \subsection{Pointwise representation of $\loglap u(x)$.}
 \label{sec:pointw-repr-logl}

Suppose that $u$ is Dini-continuous at a point $x\in \R^N$. For the proof of (\ref{eq:characterization-loglap-pointwise}), we require the following
two lemmas.

 \begin{lemma}
   \label{prelim-asympt}
   We have
   \begin{equation}
     \label{eq:claim-prelim-asympt}
       c_N \int_{B_1 \setminus B_t} \left(|x|^{2}+t^2\right)^{-\frac{N}{2}}\,dx
       = - 2 \ln t+ q_N+o(1) \qquad \text{as $t \to 0^+$,}
   \end{equation}
   where
   \begin{displaymath}
     q_N=2  \int_1^{+\infty}\big( (r^2+1)^{-\frac{N}{2}}-r^{-N}\big)r^{N-1}\,dr.
 \end{displaymath}
 \end{lemma}

 \begin{proof}
We have 
\begin{align*}
  c_N \int_{B_1\setminus B_t} \left(|x|^{2}+t^2\right)^{-\frac{N}{2}}
  \,d x 
  &=c_N \int_{B_{\frac{1}t}\setminus B_1}
    \left(|z|^{2}+1\right)^{-\frac{N}{2}}   \,d z\\
  &= c_N \int_{B_{\frac{1}t}\setminus B_1}  |z| ^{-N } \,dz 
   + c_N \int_{B_{\frac{1}t}\setminus B_1} 
    \big(\left(|z|^{2}+1\right)^{-\frac{N}{2}}- |z| ^{-N } \big)  \,d z\\
   &=- 2 \ln t+ 2 \biggl(\int_1^{+\infty}\big( (r^2+1)^{-\frac{N}{2}}-r^{-N}\big)r^{N-1}dr 
  \\& \qquad\qquad\qquad\quad -\int_{\frac{1}t}^{+\infty}
      \big( (r^2+1)^{-\frac{N}{2}}-r^{-N}\big)r^{N-1}dr\biggr),
\end{align*}
where, for $0<t \ll 1$
\begin{align*} 
  0<&\int_{\frac{1}t}^{+\infty}\big(r^{-N}- (r^2+1)^{-\frac{N}{2}} \big)
   r^{N-1}dr=\int_{\frac{1}t}^{+\infty}\big(1- (1+r^{-2})^{-\frac{N}{2}} \big)r^{-1}dr\\
  &\leq  N  \int_{\frac{1}t}^{+\infty} r^{-3}dr= \frac{N}{2} t^2 \:\to
    \: 0 \qquad \text{as $t \to 0^+$.}
\end{align*}
Consequently,
\begin{displaymath}
\lim_{t \to 0^+} c_N \Bigl(\int_{B_1 \setminus B_t}
\left(|x|^{2}+t^2\right)^{-\frac{N}{2}}   
\,d x + 2 \ln t\Bigr) = q_n
\end{displaymath}
with
\begin{displaymath}
q_N = 2 \int_1^{+\infty}\big( (r^2+1)^{-\frac{N}{2}}-r^{-N}\big)r^{N-1}dr,
\end{displaymath}
as claimed.
\end{proof}

\begin{lemma}
  \label{computations-of-constants}
  Let 
  \begin{displaymath}
    \tilde q_N = 2 \int_0^1(1+t)^{-\frac{N}{2}} t^{\frac{N}{2}-1} dt.
  \end{displaymath}
  Then, the following statements hold.
  \begin{itemize}
  \item[(i)] For $m\in\N$ the Digamma function $\psi = \frac{\Gamma'}{\Gamma}$ satisfies  
    $$
\psi(m) +\gamma =\sum \limits^{m-1}_{k=1} \frac{1}{k}\qquad \text{and}\qquad  
\psi(m+\frac{1}{2})+\gamma=-2\ln2+\sum \limits^{m}_{k=1} \frac{2}{2k-1}.
$$
\item[(ii)] We have $q_1  =2\ln(1+\sqrt{2})$, $q_2 = \ln2$ and, for $N \ge 3$,
\begin{align}\label{consant-log1} 
q_{N}=\left\{
\begin{aligned}
  & -\ln 2+ \sum^{m-1}_{k=1}\frac{1}{k2^{k}}-
  \sum^{m-1}_{k=1}\frac{1}{k}  \qquad\quad  \text{for $N=2m$,}\\
 &2\ln2-2\ln(1+\sqrt{2})+\sum^{m}_{k=1} (2k-1)^{-1}2^{\frac{3-2k}{2}}-
\sum^{m}_{k=1}\frac{2}{2k-1}  \quad \text{for $ N= 2m+1$}.  
\end{aligned}
\right.
\end{align} 
\item[(iii)] We have $\tilde q_1  =2\ln(1+\sqrt{2})$, $\tilde q_2 = \ln2$ and, for $N\geq 3$,
\begin{align}\label{consant-log2} 
\tilde q_N=\left\{
\begin{aligned}
  &\ln 2- \sum^{m-1}_{k=1}\frac{1}{k2^{k}}  \qquad\quad  \text{for $N=2m$,}\\
 &2\ln(1+\sqrt{2})- \sum^{m}_{k=1} (2k-1)^{-1}2^{\frac{3-2k}{2}} \quad \text{for $ N= 2m+1$}.  
\end{aligned}
\right.
\end{align}
\end{itemize}
\end{lemma}

While the formula for the Digamma function in
Lemma~\ref{computations-of-constants}(i) is well known, we could not
find the computation of the integrals in the representation of $q_N$ and
$\tilde q_N$ given in Parts (ii) and (iii).  These will follow from
elementary but quite lengthy computations which we will provide in the
appendix.\medskip

We may now complete the

\begin{proof}[Proof of (\ref{eq:characterization-loglap-pointwise})]
We have
\begin{align*}
  2\,w_u(x,t)+ 2 u(x)\,\ln t &
  = c_N \Bigl(\int_{\R^N\setminus B_1(x) } 
\left(|x-\tilde x|^{2}+t^2\right)^{-\frac{N}{2}} u(\tilde x)  \,d \tilde x\\
  &+ \int_{ B_{t}(x)} \left(|x-\tilde x|^{2}+t^2\right)^{-\frac{N}{2}} u(\tilde x)  \,d \tilde x \\
  &+ \int_{B_1(x) \setminus B_{t}(x) } \left(|x-\tilde
    x|^{2}+t^2\right)^{-\frac{N}{2}} u(\tilde x)  \,d \tilde x  + 2 u(x) \ln t \Bigr),
\end{align*}
for $t> 0$. Further, since $u \in L^1_0(\R^N)$, one has that 
\begin{align}
\int_{\R^N\setminus B_1(x) }& \Big|\left(|x-\tilde x|^{2}+t^2\right)^{-\frac{N}{2}} u(\tilde x)   -
 |x-\tilde x|^{-N} u(\tilde x) \Big| \,d \tilde x \nonumber \\
&\leq Nt^2\int_{\R^N\setminus B_1(x) }  |x-\tilde x|^{-N -2} u(\tilde x) \,d \tilde x \to 0 \qquad \text{as $t \to 0^+$.} \label{dini-convergence-1}
\end{align}
Moreover, 
\begin{align}
 c_N \int_{ B_{t}(x)} \left(|x-\tilde x|^{2}+t^2\right)^{-\frac{N}{2}} u(\tilde x)  \,d \tilde x  
   &= c_N \int_{ B_1} \left(| \tilde z|^{2}+1\right)^{-\frac{N}{2}} u(x+t\tilde z)  \,d \tilde z \nonumber\\
&=  c_N \int_{B_1} \left(| \tilde z|^{2}+1\right)^{-\frac{N}{2}} d\tilde z \ u(x)(1+o(1))\nonumber
\\&\to  \tilde q_N\, u(x) \qquad  \text{ as}\ \,  t \to 0^+, \label{dini-convergence-2}
\end{align}
where $\tilde q_N = 2 \int_0^1(1+t)^{-\frac{N}{2}} t^{\frac{N}{2}-1} dt$.
 In addition,
\allowdisplaybreaks
\begin{align}
&\int_{B_1(x) \setminus B_{t}(x) } \left(|x-\tilde
  x|^{2}+t^2\right)^{-\frac{N}{2}} u(\tilde x)  
\,d \tilde x + 2 u(x) \ln t \nonumber\\  
  &=\int_{B_1(x) \setminus B_{t}(x) } \left(|x-\tilde
    x|^{2}+t^2\right)^{-\frac{N}{2}} (u(\tilde x)-u(x))  \,d \tilde x
\nonumber \\
  &\qquad + u(x) \Bigl(\int_{B_1(x) \setminus B_{t}(x) }
    \left(|x-\tilde x|^{2}+t^2\right)^{-\frac{N}{2}}  \,d \tilde x + 2\, \ln t\Bigr) \nonumber\\
  &=\int_{B_1(x) \setminus B_{t}(x) } \left(|x-\tilde
    x|^{2}+t^2\right)^{-\frac{N}{2}} (u(\tilde x)-u(x))  \,d \tilde x
  + u(x)(q_N+ o(1)) \label{dini-convergence-3}
\end{align}
by Lemma~\ref{prelim-asympt}, where  
\begin{align}
\int_{B_1(x) \setminus B_{t}(x) }& \left(|x-\tilde
  x|^{2}+t^2\right)^{-\frac{N}{2}} (u(\tilde x)-u(x))  \,d \tilde x \nonumber\\
&
\to \int_{B_1(x)} |x-\tilde x|^{-N} (u(\tilde x)-u(x))  \,d \tilde x\quad {\rm as}\ t\to0^+.\label{dini-convergence-4}
\end{align}
Indeed, this follows from the dominated convergence theorem since
$$
\int_{B_1(x)} |x-\tilde x|^{-N} |u(\tilde x)-u(x)|  \,d \tilde x 
\le \omega_N \int_{0}^1 r^{-1} \max_{\tilde x \in S_r(x)}|u(\tilde x)-u(x)|dr < \infty
$$
by the Dini continuity of $u$ at $x$. Combining the asymptotics (\ref{dini-convergence-1}),~(\ref{dini-convergence-2}),~(\ref{dini-convergence-3}) and~(\ref{dini-convergence-4}), we
get
\begin{equation}
  \label{expansion 1-1}
  \begin{split}
    -2\Bigl(w_u(x,t)+ u(x)\, \ln t\Bigr) &= c_N \Bigl( \int_{B_1(x)} |x-\tilde x|^{-N} 
(u(x)-u(\tilde x))  \,d \tilde x \\
 &\hspace{3cm}- \int_{\R^N\setminus B_1(x) } |x-\tilde x|^{-N} u(\tilde x)  \,d
   \tilde x\Bigr)(1+o(1))\\ 
   &\hspace{4cm} -  \Bigl(q_N + \tilde q_N\Bigr)u(x)(1+o(1))\\
  &= \Big( \loglap u(x) -\rho_N u(x) - \Bigl(q_N + \tilde q_N\Bigr)u(x)\Big)(1+o(1))  
  \end{split}
\end{equation}
as $t\to 0^+$. Now, recall that
$\rho_N = 2\ln(2)+\psi(\frac{N}{2}) - \gamma$ with the digamma
function $\psi$.  Thus, the formula~(\ref{eq:characterization-loglap-pointwise})
follows from~(\ref{expansion 1-1}) and the fact that
\begin{equation}
  \label{eq:label-computation-q-N-tilde-q-N}
  q_N + \tilde q_N+\rho_N  = 2(\ln 2 - \gamma),
\end{equation}
which is a consequence of the previous Lemma~\ref{computations-of-constants}. \medskip
\end{proof}

\begin{remark}
\label{uniformly-dini-continuous}  
An inspection of the proof above shows that, if $u \in L^1_0(\R^N)$ is uniformly Dini continuous on a compact set $K \subset \R^N$, then the convergence in ~(\ref{eq:characterization-loglap-pointwise}) is uniform for $x \in K$. Indeed, we only need to ensure that all convergences (\ref{dini-convergence-1}),~(\ref{dini-convergence-2}),~(\ref{dini-convergence-3}) and~(\ref{dini-convergence-4}) are uniform for $x \in K$ in this case. In the case of (\ref{dini-convergence-1}), this follows from the assumption $u \in L^1_0(\R^N)$ and the fact that 
$$
\sup_{x \in K,\, \tilde x \in \R^N\setminus B_1(x)} (1+|\tilde x|)^{N} |x-\tilde x|^{-N -2} < \infty
$$ 
Moreover, in the case of (\ref{dini-convergence-1}) and (\ref{dini-convergence-4}), it follows from the uniform Dini continuity of $u$ on $K$. Finally, in the case of~(\ref{dini-convergence-3}), it merely follows from Lemma~\ref{prelim-asympt} and the boundedness of $u$ on $K$. 
\end{remark}

 \subsection{Representation of $\loglap u$ in distributional sense}
\label{sec:repr-logl-u}

Now, we are in the position to prove the distributional representation of $\loglap u$ for $u \in L^1_0(\R^N)$ 
by (\ref{eq:characterization-loglap}), where, as before $w=w_u$ is defined by the Poisson integral formula (\ref{eq:poisson-formula-wu}).

More precisely, we shall prove the following.
 \begin{proposition}
   \label{main-distr-sense}
   For $u \in L^1_0(\R^N)$ and $\phi \in C_c^D(\R^N)$, we have 
\begin{equation}
  \label{eq:distributional-limit-dini}
\int_{\R^N}u \loglap \phi dx =  2(\ln 2-\gamma) \int_{\R^N}u \phi\,dx - 
  2\lim_{t \to 0^+}  \int_{\R^N}\Bigl(w_u(x,t) + u \ln t\Bigr)\phi (x) \,dx.
\end{equation}
with $w_u$ defined by (\ref{eq:poisson-formula-wu}). Here, as before, $C_c^D(\R^N)$ denotes the space of uniformly Dini
continous functions on $\R^N$ with compact support.
\end{proposition}

 \begin{proof}
Let $w_\phi: \R^{N+1}_+ \to \R$ be given by (\ref{eq:poisson-formula-wu}) with $u$ replaced by $\phi$, i.e.,
$$
w_\phi(x,t)= \frac{c_N}{2}   \int_{\R^N} \left(|x-\tilde x|^{2}+t^2\right)^{-\frac{N}{2}}\phi(\tilde x)\,d \tilde x.
$$
We then have, for $t>0$, by Fubini's theorem
\begin{align}
  \int_{\R^N}& (w_u(x,t)\phi(x)\,dx = \frac{c_N}{2} \int_{\R^N}
\int_{\R^N} \left(|(x -\tilde x|^{2}+t^2\right)^{-\frac{N}{2}}u(x) \phi(\tilde x)\,d \tilde x dx \nonumber\\ 
  &= \frac{c_N}{2} \int_{\R^N}
    \int_{\R^N} \left(|(x -\tilde x|^{2}+t^2\right)^{-\frac{N}{2}} \phi(x)u(\tilde x) \,dx d \tilde x = \int_{\R^N} (w_\phi(x,t)u(x)\,dx. \label{distributional-limit-proof-1}  
\end{align}
Consequently, using the fact that we have already proved the pointwise representation (\ref{eq:characterization-loglap-pointwise}) for the Dini continuous function $\phi$ in place of $u$, we find that  
\begin{align*}
  \int_{\R^N}u \loglap \phi dx &= 2(\ln 2- \gamma) \int_{\R^N}\phi u\,dx - 2 \int_{\R^N}u(x) \lim_{t \to 0^+} (w_\phi(x,t) + 2 \phi(x) \ln t) \,dx\\
&= 2(\ln 2- \gamma) \int_{\R^N}\phi u\,dx - 2 \lim_{t \to 0^+} \int_{\R^N} u(x)  (w_\phi(x,t) + 2 \phi(x) \ln t) \,dx\\
                               &= 2(\ln 2-\gamma)\int_{\R^N} u\phi\,dx - \lim_{t \to 0^+} \int_{\R^N}(w_u(x,t)+2 u(x) \ln t) \phi(x)\,dx.
\end{align*}
Here we used the dominated convergence theorem in the second equality and \eqref{distributional-limit-proof-1} in the third inequality. To apply the dominated convergence theorem, we need to prove that
\begin{equation}
  \label{eq:final-estimated-needed}
|w_\phi(x,t) + 2 \phi(x) \ln t| \le C (1+|x|)^{-N} \qquad \text{for all $x \in \R^{N}$, $t \in (0,1)$}  
\end{equation}
with a constant $C>0$, and to combine this fact with the assumption $u \in L^1_0(\R^N)$. To show \eqref{eq:final-estimated-needed}, we let $R>1$ with $\supp \,\phi \subset B_R$. Then for $x \in \R^N$ with $|x| \ge 2R$ and $t \in (0,1)$ we have $2\phi(x) \ln t = 0$ and therefore 
\begin{align*}
|w_\phi(x,t) + 2 \phi(x) \ln t| = |w_\phi(x,t)| &\le  \frac{c_N}{2}   \int_{B_R} |x-\tilde x|^{-N}|\phi(\tilde x)|\,d \tilde x\\
& \le 2^{N-1}c_N |x|^{-N} \le 2^{2N-1}c_N (1+|x|)^{-N}.   
\end{align*}
Moreover, since $\phi$ is uniformly Dini continuous on $\R^N$, Remark~\ref{uniformly-dini-continuous} implies that  
$|w_\phi(x,t) + 2 \phi(x) \ln t|$ remains uniformly bounded for $x \in B_{2R}$, $t \in (0,1)$. This shows (\ref{eq:final-estimated-needed}) and finishes the proof of the proposition.
 \end{proof}

 \section{An application: The weak unique continuation property for the
   logarithmic Laplacian}
\label{sec:an-application:-weak}

This section is dedicated to an application of the theory developed in
the previous sections. More precisely, thanks to the extension
problem~(\ref{eq:distributional-2-extension-thm2}) in $\R^{N+2}$, we are able to prove the 
following weak unique continuation property for the logarithmic Laplacian.

\begin{theorem}\label{unique-cont}
  Let $u \in L^1_0(\R^N)$. If there is an
  nonempty open subset $\Omega \subseteq \R^N$ such that
  \begin{equation}
    \label{eq:assumption-unique-cont}
    u \equiv 0 \quad \text{in $\Omega$} \quad \quad \text{and} \quad
    \quad \loglap u \equiv 0 \quad \text{in $\Omega$ in distributional sense,}
  \end{equation}
  Then $u \equiv 0$ on $\R^N$.
\end{theorem}

\begin{proof}
Let $W_u$ be the two-dimensional extension on $\R^{N+2}$ given by (\ref{eq:poisson-formula-loglap-2-dim}), which has the properties of Theorem~\ref{thm:main2}. Moreover, consider the domain 
\begin{displaymath}
  U:= \Bigl(\Omega \times \{0_{\R^2}\}\Bigr) \cup \Bigl(\R^{N+2} 
\setminus \bigl(\R^N \times \{0_{\R^2}\}\bigr)\Bigr) \subseteq  \R^{N+2}.
\end{displaymath}
 We first show that $W_u$ is harmonic in $U$. For this, we let 
$\phi \in C^\infty_c(U)$. Then, by~(\ref{eq:distributional-2-extension-thm2}), we have
  \begin{align*}
  \int_{U}W_u (-\Delta \phi)\,dxdy  &=   \int_{\R^{N+2}}W_u (-\Delta \phi)\,dxdy  \\
  &= 2\pi \int_{\R^{N}}u(x)\phi(x,0) \,dxdy =2\pi \int_{\Omega}u(x)\phi(x,0) \,dxdy  = 0.  
  \end{align*}
  Hence $W_u$ is harmonic in distributional sense in $U$ and therefore
  also harmonic in $U$ in classical sense by Weyl's lemma.  Moreover,
  for $\psi \in C^\infty_c(\Omega) \subset C^\infty_c(\R^N)$ we have, by hyopothesis~\eqref{eq:assumption-unique-cont} and
  (\ref{eq:characterization-loglap-2-dim}), that
  \begin{align*}
    0  = \int_{\R^N} u \loglap \phi \,dx &= 2(\ln 2-\gamma)\int_{\R^N}
                                            u\phi\,dx - 
\lim_{|y| \to 0} \int_{\R^N} 2(W_u(x,y)+ u(x)\ln |y|) \phi(x)\,dx\\
 &= 2(\ln 2-\gamma)\int_{\Omega} u\phi\,dx - \lim_{|y| \to 0}
   \int_{\Omega} 2(W_u(x,y)+ u(x)\ln |y|) \phi(x)\,dx\\
 &= - \lim_{|y| \to 0} \int_{\Omega} W_u(x,y) \phi(x)\,dx
 \\&= - \int_{\Omega} W_u(x,0) \phi(x)\,dx.
  \end{align*}
  Consequently, $W_u(\cdot,0) \equiv 0$ in $\Omega$. Hence for $x_0 \in \Omega$ we have
  $$
  0 = \Delta W_u(x_0,0)= \Delta_x W_u(x_0,0) + \Delta_{y} W_u(x_0,0)= \Delta_{y} W_u(x_0,0)
  $$
  and
  $$
  0 = \Delta^2 W_u(x_0,0)= \Delta_x^2 W_u(x_0,0) +2 \Delta_x  
  \Delta_y W_u(x_0,0)+ \Delta_y^2 W_u(x_0,0) = \Delta_y^2 W_u(x_0,0).
  $$
  Inductively, we get
  $$
  \Delta_y^k W_u(x_0,0) = 0 \qquad \text{for every $x_0 \in \Omega$, $k \in \N$.}
$$
Since $W_u(x_0,\cdot)$ is a radial function in $y$, this implies that
all partial derivatives with respect to the $y$-variable vanish at
$z_0$, whereas also all partial derivatives with respect to the $x$-variable
vanish since $W_u(\cdot,0) \equiv 0$ in $\Omega$. Hence all partial
derivatives of $W_u$ vanish at points $(x_0,0)$ with $x_0 \in \Omega$.
Since $W_u$ is analytic on the connected open set $U$, it follows that
$W_u \equiv 0$ on $U$. But then we have
$$
u= -\lim_{y \to 0} \frac{W_u(\cdot,y)}{\ln |y|} =0 \qquad \text{in $L^1_{loc}(\R^N)$}
$$
by (\ref{u-limit-L-1-loc-2-dim}), as claimed.
\end{proof}

\appendix
\section{Appendix }

Her we give the 
 {\bf Proof of  Lemma~\ref{computations-of-constants} Parts (ii) and (iii).}  We recall that
$$
q_N =2\int_1^{+\infty}\big( (r^2+1)^{-\frac{N}{2}}-r^{-N }\big)r^{N-1}dr=\int_1^{+\infty}\big( (t+1)^{-\frac{N}{2}}-t^{-\frac{N}{2}}\big)t^{\frac{N}{2}-1}dt
$$
and
$$
\tilde q_N =2\int_0^1 (r^2+1)^{-\frac{N}{2}}r^{N-1}dr=\int_0^1 (t+1)^{-\frac{N}{2}}t^{\frac{N}{2}-1}dt.
$$

{\bf  Claim 1:}  $\tilde q_{1}=2\ln(1+\sqrt{2})$, $\tilde q_{2}= \ln2$ and, for $N\geq 3$, 
\begin{align}\label{consant-log3} 
\tilde q_{N}=\left\{
\begin{aligned}
  &\ln 2- \sum^{m-1}_{k=1}\frac{1}{k2^{k}}  \qquad\quad  \text{for $N=2m$,}\\
 &2\ln(1+\sqrt{2})- \sum^{m}_{k=1} (2k-1)^{-1}2^{\frac{3-2k}{2}} \quad \text{for $ N= 2m+1$}.  
\end{aligned}
\right.
\end{align} 

\begin{proof}
When $N=2$, 
$$\tilde q_{2}=\int_0^1(1+t)^{-1} dt=\ln2.$$
When $N=4$,   \begin{align*}
 \tilde q_{4} & = \int_0^1 (1+t)^{-2} t dt  
 \\[2mm]&= - \Big( (1+t)^{-1}t\Big|^1_0 -\int_0^1(1+t)^{-1} dt\Big)=  \ln2-\frac{1}{2}.
 \end{align*} 
When $N=6$,   \begin{align*}
 \tilde q_{6} & = \int_0^1 (1+t)^{-3} t^2 dt  
 \\[2mm]&= -\frac{1}{2}  \Big( (1+t)^{-2}t^2\Big|^1_0 -2\int_0^1(1+t)^{-2}t dt\Big)=  \ln2-\frac{1}{2}-\frac{1}{8}.
 \end{align*} 
When $N=2m$, 
\begin{align*}
 \tilde q_{2m} & = \int_0^1 (1+t)^{-m} t^{m-1} dt  
 \\[2mm]&= - \frac{1}{m-1} \Big( (1+t)^{1-m}t^{m-1}\Big|^1_0 -\int_0^1(1+t)^{1-m}t^{m-1} dt\Big)=  \tilde \varrho_{2(m-1)}- \frac{1}{(m-1)2^{m-1}}. 
 \end{align*} 
 For $m\geq 1$, we have that  $N=2m$
$$\tilde q_{N}= \ln 2- \sum^{m-1}_{k=1}\frac{1}{k2^{k}}.$$
\bigskip

When $N=1$, 
$$\tilde q_{1}= \int_0^1 (1+t)^{-\frac{1}{2}} t^{-\frac{1}{2}}  dt=2\int_0^1(1+r^2)^{-1} dr =2\ln(1+\sqrt2).$$ 
When $N=3$,   \begin{align*}
 \tilde q_{3} & = \int_0^1 (1+t)^{-\frac{3}{2}} t^{\frac{1}{2}}  dt
 \\[2mm]&= -2\Big( (1+t)^{-\frac{1}{2}}t^{\frac{1}{2}}\Big|^1_0 -\frac{1}{2}\int_0^1(1+t)^{-\frac{1}{2}}t^{-\frac{1}{2}} dt\Big)= 2\ln(1+\sqrt2)- 2^{\frac{1}{2}}.
 \end{align*} 
When $N=5$,   \begin{align*}
 \tilde q_{5} & = \int_0^1 (1+t)^{-\frac{5}{2}} t^{\frac{3}{2}} dt  
 \\[2mm]&= -\frac{2}{3} \Big( (1+t)^{-\frac{3}{2}}t^{\frac{3}{2}}\Big|^1_0 -\frac{3}{2}\int_0^1(1+t)^{-\frac{3}{2}}t^{\frac{1}{2}} dt\Big)= 2\ln(1+\sqrt2)- 2^{\frac{1}{2}}-\frac{1}{3} 2^{-\frac{1}{2}}. 
 \end{align*} 
When $N=2m+1$, 
\begin{align*}
 \tilde q_{2m+1} & = \int_0^1 (1+t)^{-\frac{2m+1}{2}} t^{\frac{2m-1}{2}} dt  
 \\[2mm]&= -\frac{2}{2m-1} \Big( (1+t)^{-\frac{2m-1}{2}} t^{\frac{2m-1}{2}}\Big|^1_0 -\frac{2m-1}{2}\int_0^1(1+t)^{-\frac{2m-1}{2}} t^{\frac{2m-1}{2}} dt\Big)
 \\[2mm]& =  \tilde \varrho_{2(m-1)}-\frac{2}{(2m-1)2^{-\frac{2m-1}2}}.
 \end{align*} 
 For $m\geq 1$, we have that  $N=2m+1$
$$\tilde q_{N}=2\ln(1+\sqrt{2})-\sum^{m}_{k=1} (2k-1)^{-1}2^{\frac{3-2k}{2}}.$$
{\bf Claim 1} is proved. 
\end{proof}

 {\bf  Claim 2:}  $q_{1}=2\ln(1+\sqrt{2})$, $q_{2}= \ln2$ and, for $N\geq 3$,
\begin{align}\label{consant-log4} 
q_{N}=\left\{
\begin{aligned}
  & -\ln 2+ \sum^{m-1}_{k=1}\frac{1}{k2^{k}}- \sum^{m-1}_{k=1}\frac{1}{k}  \qquad\quad  \text{for $N=2m$,}\\
 &2\ln2-2\ln(1+\sqrt{2})+\sum^{m}_{k=1} (2k-1)^{-1}2^{\frac{3-2k}{2}}-\sum^{m}_{k=1}\frac{2}{2k-1}  \quad \text{for $ N= 2m+1$}.  
\end{aligned}
\right.
\end{align} 

\begin{proof}
When $N=2$, 
$$q_{2}= \int_1^{+\infty}\big( (t+1)^{-1}-t^{-1}\big)  dt=-\ln2.$$
When $N=4$,   \begin{align*}
 q_{4} & =\lim_{M\to+\infty}  \int_1^M\big( (t+1)^{-2}-t^{-2}\big)t dt 
 \\[2mm]&= - \Big[ \big((1+t)^{-1}-t^{-1}\big)t\Big|^{+\infty}_1 -\int_1^{+\infty}\big( (t+1)^{-1}-t^{-1}\big)  dt\Big]= - \ln2+ (2^{-1}-1).
 \end{align*} 
When $N=2m$, 
\begin{align*}
 q_{2m} & = \lim_{M\to+\infty}  \int_1^M \big( (t+1)^{-m}-t^{-m}\big)t^{m-1}dt
 \\[2mm]&= - \frac{1}{m-1} \Big[ \big((1+t)^{1-m}-t^{1-m}\big)t^{m-1}\Big|_1^{+\infty} -(m-1)\int_1^{+\infty}\big( (t+1)^{-m}-t^{-m}\big)t^{m-1}dt\Big]
  \\[2mm]& =  \tilde \varrho_{2(m-1)}+  \frac{1}{(m-1)2^{m-1}}-\frac{1}{m-1} . 
 \end{align*} 
 For $m\geq 2$, we have that  $N=2m$
$$\tilde q_{N}= -\ln 2+ \sum^{m-1}_{k=1}\frac{1}{k2^{k}}- \sum^{m-1}_{k=1}\frac{1}{k}.$$
\bigskip

When $N=1$, 
\begin{align*}
q_{1} = \int_1^{+\infty}\big( (t+1)^{-\frac{1}{2}}-t^{-\frac{1}{2}}\big) t^{-\frac{1}{2}} dt 
 & =2\int_1^{+\infty}\big((1+r^2)^{-\frac{1}{2}}-\frac{1}r\big) dr
\\[2mm]&=2\Big(\ln \big( r+\sqrt{1+r^2}\big)-\ln r\Big)\Big|_{1}^{+\infty}
\\[2mm]& =2\ln 2-2\ln(1+\sqrt2).
\end{align*}  
When $N=3$,   
\begin{align*}
 q_{3} & = \int_1^{+\infty}\big( (1+t)^{-\frac{3}{2}}-t^{-\frac{3}{2}}\big) t^{\frac{1}{2}}  dt
 \\[2mm]&= -2\Big[ \big( (1+t)^{-\frac{1}{2}}-t^{-\frac{1}{2}}\big) t^{\frac{1}{2}}\Big|_{1}^{+\infty} -\frac{1}{2}\int_1^{+\infty}\big( (t+1)^{-\frac{1}{2}}-t^{-\frac{1}{2}}\big) t^{-\frac{1}{2}} dt \Big]
  \\[2mm]&=2\ln 2-2\ln(1+\sqrt2) +2  \big( 2^{-\frac{1}{2}}-1\big).
 \end{align*} 
When $N=2m+1$, 
\begin{align*}
 q_{2m+1} & = \int_1^{+\infty}   \big( (1+t)^{-\frac{2m+1}{2}}-t^{-\frac{2m+1}{2}}\big) t^{\frac{2m-1}{2}}  dt 
 \\[2mm]&= -\frac{2}{2m-1} \Big[ \big( (1+t)^{-\frac{2m-1}{2}}-t^{-\frac{2m-1}{2}}\big) t^{\frac{2m-1}{2}}\Big|_{1}^{+\infty} -\frac{2m-1}{2}\int_1^{+\infty}   \big( (1+t)^{-\frac{2m-1}{2}}-t^{-\frac{2m-1}{2}}\big) t^{\frac{2m-3}{2}}  dt \Big]
 \\[2mm]& =  q_{2(m-1)}+\Big(\frac{2}{(2m-1)2^{-\frac{2m-1}2}}-\frac{2}{2m-1}\Big).
 \end{align*} 
 For $m\geq 1$, we have that  $N=2m+1$
$$q_{N}=2\ln 2-2\ln(1+\sqrt2) +\sum^{m}_{k=1} (2k-1)^{-1}2^{\frac{3-2k}{2}}-\sum^{m}_{k=1}\frac{2}{2k-1}.$$
The proof of {\bf Claim 2} is complete.
\end{proof}

%
%
%
%
%
%

\end{document}